\documentclass[12pt]{article}

\usepackage{graphicx}    
\usepackage{amsmath} 
\usepackage{xcolor}
\usepackage{amssymb}
\usepackage{amsthm} 
\usepackage{amsfonts}
\usepackage{booktabs}
\usepackage{hyperref}    
\usepackage{geometry}    
\usepackage{caption}     
\usepackage{cite}
\usepackage{algorithm}
\usepackage{algorithmicx}
\usepackage{algpseudocode}
\usepackage{pgfplots}
\usepackage{bm}
\usepackage{float}
\usetikzlibrary{patterns,positioning}
\usepackage{tikz}
\usetikzlibrary{decorations.markings}
\usetikzlibrary{matrix}
\usepackage{pgfplots}
\usetikzlibrary{shapes, arrows.meta, patterns}
\usepackage{mathrsfs}

\usepackage{subcaption}

\newcommand{\bfa}[1]{\boldsymbol{#1}} 			
			
\newcommand{\bfeps}{\boldsymbol{\epsilon}}

\newcommand{\tr}{\text{tr}}       				%

\newcommand{\jump}[1]{[\![#1]\!]}

\DeclareMathAlphabet{\mathpzc}{OT1}{pzc}{m}{it}

\newcommand{\bfn}{\boldsymbol{n}}	
	
\newcommand{\bfu}{\boldsymbol{u}}

\newcommand{\bfw}{\boldsymbol{w}}

\newcommand{\bfB}{\boldsymbol{B}}	
\newcommand{\bfx}{\boldsymbol{x}}

\newcommand{\bfT}{\boldsymbol{T}}		
\newcommand{\bfI}{\boldsymbol{I}}	 
\newcommand{\bfzero}{\boldsymbol{0}}

\newcommand{\bff}{\boldsymbol{f}}	
\newcommand{\bfg}{\boldsymbol{g}}

\newtheorem{Proposition}{Proposition}
\usepackage{authblk}

\usepackage{etoolbox}

\newtheorem{Algorithm}{Algorithm}
\newtheorem{Theorem}{Theorem}
\newtheorem{remark}{Remark}

\newtheorem*{cwf}{Continuous weak formulation}
\newtheorem*{dwf}{Discrete formulation}

\providecommand{\keywords}[1]
{
  \small	
  \textbf{\textit{Keywords---}} #1
}

\title{ An $hp$-adaptive finite element framework for static cracks: The impact of pointwise density variations on mode I, mode II, and mixed-mode fracture}

\author[]{S. M. Mallikarjunaiah}

\affil[]{Department of Mathematics \& Statistics, Texas A\&M University-Corpus Christi, TX- 78412, USA}
\affil[ ]{\textit{E-mail addresses:} \texttt{M.Muddamallappa@tamucc.edu}}
\date{}

\begin{document}

\maketitle  

\begin{abstract}
A robust $hp$-adaptive finite element framework is presented for the investigation of static cracks in materials characterized by complex, pointwise density variations. Within such heterogeneous media, the equilibrium equation—governed by the divergence of the stress tensor—is reduced to a vector-valued quasilinear partial differential equation, wherein significant gradients and nonlinearities are introduced into the governing constitutive relations by spatial density variations. To ensure that these localized field variables are accurately captured, an automated $hp$-refinement strategy is implemented so that element sizes ($h$) and polynomial degrees ($p$) are concurrently optimized.  A dual-indicator approach is employed by the framework. In this approach, $h$-refinement is driven by the \textit{Kelly error estimator}, which utilizes the jump of the normal derivative across element interfaces, while $p$-refinement is determined by the decay of \textit{Legendre-Fourier coefficients} to assess the local smoothness of the solution. Optimal convergence rates are ensured by this dual-refinement strategy, particularly within the singular regions surrounding the crack tip where material heterogeneities and geometric discontinuities intersect. The performance of the proposed method is rigorously evaluated across three distinct loading regimes: pure tensile (Mode I), pure shear (Mode II), and mixed-mode loading. It is demonstrated through numerical simulations that the intricate interaction between the fracture process zone and the underlying material heterogeneities is effectively resolved by the $hp$-adaptive scheme. Furthermore, the influence of specific density-dependent constitutive relations on the resulting crack-tip stress fields and strain energy density distributions is analyzed in detail. Superior computational efficiency is shown to be provided by the framework, which offers a high-fidelity computational tool suitable for the predictive design and failure analysis of functionally graded and heterogeneous engineering materials.
\end{abstract}

\vspace{.1in}

\noindent \keywords{$hp$-adaptive finite element method; Pointwise density variations; Constitutive relations; Heterogeneous media; Crack-tip singularities; Strain energy density}
 
\section{Introduction}
The Finite Element Method (FEM) has solidified its position as the preeminent numerical paradigm for the discretization of partial differential equations (PDEs) governing complex boundary value problems (BVPs) in solid mechanics \cite{zienkiewicz2005finite,bonito2020finite}. Historically, the quest for numerical precision has been dominated by two distinct refinement strategies: $h$-refinement, which achieves accuracy through the iterative subdivision of the mesh into smaller topological units, and $p$-refinement, which elevates the polynomial order of the basis functions within a static mesh. While $h$-refinement is robust for a wide range of engineering applications, it is often plagued by a slow algebraic convergence rate, particularly when the solution exhibits localized singularities or sharp internal layers. Conversely, $p$-refinement offers spectral-like convergence for globally smooth solutions but remains notoriously inefficient when confronted with the non-smooth features characteristic of crack-tip neighborhoods. To address these bifurcated limitations, the $hp$-adaptive finite element method ($hp$-AFEM) emerged as a sophisticated hybrid framework \cite{babuska1981p,babuvska1986basic}. By concurrently optimizing the characteristic element size ($h$) and the local polynomial degree ($p$), $hp$-AFEM enables the method to achieve exponential rates of convergence in the energy norm, provided the solution possesses a certain level of piecewise analytical regularity \cite{schwab1998p,szabo1986mesh}. In the context of nonlinear problems, such as those governed by algebraically complex constitutive laws, the computational efficiency of $hp$-AFEM is not merely an advantage but a necessity. Because each iteration of a nonlinear solver (e.g., Newton-Raphson) incurs significant computational costs, the ability to reach a prescribed error tolerance with a drastically reduced number of degrees of freedom (DOFs) compared to traditional $h$-refinement is essential for maintaining computational tractability.

In the domain of fracture mechanics, the numerical challenge is exacerbated by the presence of crack-tip singularities. Classical Linear Elastic Fracture Mechanics (LEFM), though foundational, is predicated on the assumption of infinitesimal deformations and a linear mapping between the Cauchy stress and linearized strain tensors. This framework inherently predicts a square-root singularity in the stress and strain fields at the crack tip, a result that leads to physically implausible, unbounded energy densities \cite{Inglis1913,love2013treatise}. Such issues have prompted a paradigm shift toward nonlinear elasticity and the development of constitutive representations that more accurately reflect the finite capacity of materials to undergo deformation. While numerous models—including gradient-enhanced and non-local theories—have been proposed to regularize these singularities, they often introduce substantial computational overhead or require the calibration of non-physical length-scale parameters \cite{gurtin1975,gazonas2023numerical}.  The historical trajectory of constitutive modeling has been defined by a persistent attempt to bridge the gap between mathematical simplicity and physical reality. Traditionally, the Cauchy formulation and the Green-elastic (hyperelastic) formulation have provided the strong foundation for structural mechanics \cite{zienkiewicz2005finite}. However, these frameworks are intrinsically limited by their inability to accommodate classes of material response where the strain remains small as stress magnitudes grow without bound. This is most palpable in brittle fracture, where classical models often fail to prevent non-physical strain singularities.

The quest for regularization has led to a diverse array of nonlinear strategies. Knowles and Sternberg \cite{knowles1983large} investigated hyperelastic materials exhibiting hardening/softening responses, while Tarantino \cite{tarantino1997nonlinear} utilized the ``Bell constraint'' to govern volumetric response. A transformative shift occurred with the introduction of \textit{Implicit Constitutive Theory} by Rajagopal \cite{rajagopal2003implicit,rajagopal2011non}. This theory defines the mechanical response through implicit relations of the form $\mathbf{f}(\bfT, \bfeps) = \mathbf{0}$, where $\bfT$ is the Cauchy stress and $\bfeps$ is the linearized strain measure.  Within this framework, ``strain-limiting'' models express linearized strain as a bounded function of Cauchy stress, even under singular conditions \cite{bustamante2018nonlinear,gou2023computational,gou2023finite,bulivcek2015existence,rodriguez2021stretch,kulvait2013anti,yoon2022CNSNS,yoon2022MAM,gou2015modeling}. The present study extends this by incorporating pointwise variations in density. This introduces a ``multi-scale'' challenge: the $hp$-adaptive scheme must resolve the crack-tip singularity while simultaneously approximating the material's internal gradients. While previous efforts focused on pure Mode I loading \cite{yoon2024finite}, a comprehensive $hp$-adaptive treatment of the mixed-mode fracture problem in the presence of density-induced nonlinearities remains a significant gap in the current state of the art.

Rajagopal and his collaborators introduced, through a rigorous thermodynamic re-evaluation of the constitutive relations for elastic bodies \cite{rajagopal2003implicit,rajagopal2007elasticity,rajagopal2011non,rajagopal2018mechanics}. Deviating from the traditional Cauchy-Green formulations, Rajagopal’s theory introduces implicit constitutive models where the relationship between stress and deformation gradient is defined by a manifold rather than a simple explicit function \cite{bustamante2018nonlinear}. A particularly salient subclass of these models is the "strain-limiting" material, which allows for the representation of linearized strain as a uniformly bounded function across the entire domain, even when the stress becomes singular \cite{rajagopal2011modeling}. This "limiting strain" property is critical for the investigation of brittle fracture, as it bridges the gap between the theoretically singular stress fields required for crack propagation and the physically bounded strain fields \cite{MalliPhD2015,mallikarjunaiah2015direct}.

Despite the theoretical elegance of Rajagopal’s framework \cite{rajagopal2007elasticity}, its application to heterogeneous media---specifically those characterized by pointwise density variations---remains an underdeveloped research frontier. Density heterogeneity often induces localized variations in material response to external stimuli, generating complex internal gradients that may interact unpredictably with damage or fracture process zones. In \cite{rajagopal2021implicit}, a novel theory for deriving constitutive response classes was proposed for materials exhibiting spatial density variations, such as shale, concrete, biological bone, medical implants, 3D-printed titanium, and spacecraft thermal protection systems.  A specific class of implicit relations, wherein stress and strain appear linearly, has been investigated in \cite{itou2023generalization,itou2022investigation,itou2019well,itou2021implicit}. These studies developed boundary value problems for cracked bodies subjected to non-penetration conditions, establishing well-posedness by thresholding dilatation and employing Lions' existence theorem for pseudo-monotone variational inequalities. Furthermore, the state of stress and strain near circular voids in materials described by algebraically nonlinear constitutive relations was examined in \cite{murru2021stress,murru2021astress,vajipeyajula2023stress,mallikarjunaiah2025crack}. Building on this, a novel approach to modeling damage initiation in concrete was introduced in \cite{murru2022a,murru2022b}, addressing existing limitations by defining damage through load-induced microscopic density changes. In \cite{alagappan2023note}, a model predicated on the concepts in \cite{rajagopal2021implicit} was introduced for elastic bodies with moduli dependent on both density and mechanical pressure. Numerical results indicated a distinct response compared to classical homogeneous linear elasticity. Similarly, \cite{arumugam2024new} proposed a new constitutive relation specifically for bone tissue. Despite these advancements, the aforementioned studies have not addressed crack-tip fields---such as stress, strain, and strain energy density---within this new class of materials. However, recent computational efforts \cite{gou2025computational,yoon2024finite} have begun considering static cracks in elastic solids governed by the relations proposed in \cite{rajagopal2021implicit}.

The present work addresses this critical gap by developing a robust $hp$-adaptive framework tailored for the analysis of stationary cracks in materials with spatially varying density. The core innovation of this study lies in the implementation of an automated $hp$-refinement strategy that resolves the intricate interaction between material nonlinearity and pointwise density-dependent constitutive laws. By leveraging a high-order discretization scheme, the method is shown to capture the delicate transition between high-stress concentrations and bounded strain fields with remarkable precision. The remainder of this paper is organized as follows: Section 2 provides a detailed derivation of the governing nonlinear constitutive equations and the specific density-dependent relation $\phi_1$. Section 3 outlines the development of the boundary-value problem for a static crack and discusses the existence of a solution. Section 4 describes the finite element discretization of the nonlinear problem and the $hp$-adaptive finite element formulation, including \textit{Kelly} error estimation and automated refinement algorithms. In Section 5, a comprehensive suite of numerical experiments across Mode I, Mode II, and mixed-mode loading conditions is presented. These results facilitate a parametric sensitivity analysis of crack-tip fields with respect to the material constant $\beta$. Finally, Section 6 offers concluding remarks and discusses the implications of this work for future studies.

\section{Introduction to the mathematical framework of Rajagopal's theory of elasticity}

Let $\mathcal{D} := \mathcal{D}(t) \subset \mathbb{R}^2$ be a bounded, connected polyhedral domain with a smooth boundary $\partial \mathcal{D}$ and a crack set denoted by $\Gamma$. We consider the domain to be occupied by an elastic body. Let $\mathcal{D}$ be in reference (undeformed) configuration. The boundary of this domain is assumed to be Lipschitz continuous, which ensures it is sufficiently regular for the application of standard integral theorems and the formulation of well-posed boundary value problems. The vector space of second-order symmetric tensors in two dimensions is denoted by $Sym(\mathbb{R}^{2\times2})$. This space is equipped with the standard Frobenius inner product, defined for any two tensors $\mathbf{A}_{1}, \mathbf{A}_{2} \in Sym(\mathbb{R}^{2\times2})$ as $\mathbf{A}_{1}:\mathbf{A}_{2} = \mathrm{tr}(\mathbf{A}_{1}^{T}\mathbf{A}_{2})$, where $\mathrm{tr}(\cdot)$ is the trace operator. The associated norm is given by $\mathbf{\left\Vert A\right\Vert = \sqrt{A:A}}$. The motion of the body is described by mapping each material point $\bfa{X}$ from the reference configuration $\mathcal{D}$ to its corresponding spatial position $\bfa{x}$ in the current (deformed) configuration. The {displacement vector} $\bfa{u}:\mathcal{D}  \longrightarrow \mathbb{R}^{2}$ quantifies the change in position of a material point and is defined as: $\bfa{u}(\bfa{X}) = \bfa{x}(\bfa{X}) - \bfa{X}$
The local deformation at a point is characterized by the {deformation gradient tensor} $\bfa{F}:\mathcal{D}  \longrightarrow \mathbb{R}^{2\times2}$, which is defined as the gradient of the spatial position with respect to the reference position: $
\bfa{F} = \frac{\partial\bfa{x}}{\partial\bfa{X}} = \bfa{I} + \nabla\bfa{u}$
where $\bfa{I}$ is the second-order identity tensor and $\nabla$ is the gradient operator with respect to $\bfa{X}$. From the deformation gradient, we can define objective (frame-indifferent) measures of stretching and rotation. The {right and left Cauchy-Green stretch tensors}, denoted by $\bfa{C}$ and $\bfa{B}$ respectively, are symmetric positive-definite tensors given by \cite{gurtin1982introduction}: 
$$
\bfa{C} = \bfa{F}^{T}\bfa{F} \quad \text{and} \quad \bfa{B} = \bfa{F}\bfa{F}^{T}.
$$ 
Based on these, we define two fundamental nonlinear strain measures. The {Green-St. Venant strain tensor} $\bfa{E}:\Omega \longrightarrow Sym(\mathbb{R}^{2\times2})$ measures strain with respect to the reference configuration and is expressed as: 
$$
\bfa{E} = \frac{1}{2}(\bfa{C}-\bfa{I}) = \frac{1}{2}\left(\nabla\bfa{u} + (\nabla\bfa{u})^T + (\nabla\bfa{u})^T\nabla\bfa{u}\right).
$$ 
Conversely, the {Almansi-Hamel strain tensor} $\boldsymbol{e}$ measures strain with respect to the current configuration and is defined as: 
$$
\bfa{e} = \frac{1}{2}(\bfa{I}-\bfa{B}^{-1}).
$$ 

In \cite{rajagopal2007elasticity,rajagopal2018mechanics,rajagopal2003implicit} constitutive relationship between the Cauchy stress $\bfa{T} \colon \mathcal{D} \to \mathbb{R}^{d \times d}_{sym}$ and the left Cauchy-Green strain $\bfa{B}$ is defined by the implicit relation:
\begin{equation}\label{model1}
\mathcal{F}(\bfa{T}, \;  \bfa{B}) = \bfa{0},
\end{equation}
where $\mathcal{F}$ is a nonlinear tensor-valued mapping. The governing constitutive framework stems from a general implicit relationship linking the Cauchy stress $\mathbf{T}$, the linearized strain $\bfeps$, and the material density $\rho$. Assuming that the function $\mathcal{F}$ is \textit{isotropic}, then one can write the most general form of the implicit constitutive relation \cite{rajagopal2011modeling,rajagopal2009class}:
\begin{align} 
0 &= \delta_0 \, \bfI + \delta_1 \, \bfT + \delta_2 \, \bfB +  \delta_3 \, \bfT^2 + \delta_4 \, \bfB^2 +  \delta_5 \, ( \bfT \bfB + \bfB \bfT) + {\delta_6 \, ( \bfT^2 \bfB + \bfB \bfT^2)}  \notag \\
& + \delta_7 \, ( \bfB^2 \bfT + \bfT \bfB^2) +  \delta_8 \, ( \bfT^2 \bfB^2 + \bfB^2 \bfT^2),  
\end{align}
where the material moduli $\delta_i$ for $i=0,\, 1, \,  \ldots, 8$ are scalar functions that depend on the density and the invariants for the pair $\bfT$ and $\bfB$, i.e.,
\begin{equation}
\big\{\rho, \, \tr (\bfT), \, \tr(\bfB), \, \tr(\bfT^2), \, \tr(\bfB^2), \,\tr(\bfT^3), \, \tr(\bfB^3), \,\tr(\bfT \, \bfB), \, \tr(\bfT^2 \, \bfB), \,
\tr(\bfT \, \bfB^2), \, \tr(\bfT^2 \, \bfB^2)   \big\}.
\end{equation}
Upon linearization, predicated on the assumption of infinitesimal displacement gradients, this generalized implicit response takes the form of a polynomial expansion:
\begin{equation} \label{eq:general_implicit}
    \beta_{0} \bfeps  + \beta_{1}\mathbf{I} + \beta_{2}\mathbf{T} + \beta_{3}\mathbf{T}^{2} + \beta_{4}(\mathbf{T}\bfa{\epsilon} + \bfa{\epsilon} \mathbf{T}) + \beta_{5}(\mathbf{T}^{2} \bfeps + \bfeps \mathbf{T}^{2}) = \mathbf{0}.
\end{equation}
The scalar response functions $\beta_i$ ($i=0,\dots,5$) exhibit specific dependencies: $\beta_1, \beta_2, \beta_3$ are permitted to depend linearly on the strain $\bfeps$, whereas $\beta_0, \beta_4, \beta_5$ are functions solely of the invariants of the stress tensor $\bfT$ \cite{rajagopal2021implicit}.  To derive a tractable model for porous solids, we focus on a specific subclass of this general formulation where the constitutive relation remains linear with respect to the stress tensor $\bfT$. This restriction yields a model governed by constant material moduli $E_i$ and $a_i$, expressed as:
\begin{equation} \label{eq:subclass_linear}
    (1 + a_{3} \, \tr \bfeps) \bfeps = E_{1}(1 + a_{1}\operatorname{tr}\boldsymbol{\epsilon})\mathbf{T} + E_{2}(1 + a_{2}\operatorname{tr}\boldsymbol{\epsilon})(\operatorname{tr}\mathbf{T})\mathbf{I}.
\end{equation}
Here, $\mathbf{I}$ denotes the identity tensor. It is instructive to note that standard linear elasticity is recovered as a degenerate case of this model when the coupling parameters vanish, i.e., $a_1 = a_2 = a_3 = 0$, and then we can identify 
\begin{equation}\label{eq:material_coeff}
E_1 = \dfrac{1 + \nu}{E} = \dfrac{1}{2 \, \mu} >0, \quad E_2 = - \dfrac{ \nu}{E}  <0,
\end{equation}
where $E$ is the Young's modulus and $\nu$ is the Poisson's ratio, and these are related to the Lam\`e constants $\lambda$ and $\mu$:
\begin{equation}\label{eq:linear_lame}
\lambda =  \dfrac{ E \, \nu}{(1 + \nu) (1 - 2 \nu)}, \quad \mu = \dfrac{ E }{2 (1 + \nu)}.
\end{equation}

\begin{remark}
In the regime of small displacement gradients, where $\|\nabla \bfu\|  \ll 1$, several mechanical simplifications are established for the analysis. Specifically, the Cauchy stress $T$ and the Piola-Kirchhoff stress $\bfa{S}$ are treated as equivalent, $\bfT \approx \bfa{S}$, which is a consequence of the following kinematic approximations:
\begin{equation}
    \det \bfa{F} \approx 1 + \tr(\nabla \bfu), \quad \bfa{F}^{-T} = (I + \nabla \bfu)^{-T} \approx (I - \nabla \bfu)^{T}.
\end{equation}
These assumptions align with the classical theory of linearized elasticity. It is important to note that while $\text{tr}(\nabla u)$ is ignored when compared to $1$ because it is infinitesimal, the term $a_{\ast} \text{tr}(\nabla u)$ must be retained as the material parameter $a_{\ast}$ is large. Consequently, the balance of mass is approximated by:
\begin{equation}
    \rho \approx \rho_{R}(1 - a_{\ast} \,  \tr(\nabla \bfu)).
\end{equation}
Given the small norm of the displacement gradient, although $\rho \approx \rho_{R}$, the term $a_{\ast}  \text{tr} \epsilon$ is non-negligible and is preserved in the constitutive framework.
\end{remark}
A critical feature of this implicit model is revealed by taking the trace of Eq. \eqref{eq:subclass_linear}. This operation yields a scalar constraint equation, $\Phi( \tr \bfeps, \;  \tr \bfT) =0$, which implicitly couples the volumetric strain and the mean normal stress: 
\begin{equation} \label{eq:trace_implicit}
   ( 1 + a_3 \,  \tr \bfT) \,  \tr \bfeps - E_1 (1 + a_1 \, \tr \bfeps) \, \tr \bfT - d \, E_2 (1 + a_2 \, \tr \bfeps) =0
\end{equation}
This algebraic relation can be inverted to provide an explicit definition of the volumetric strain in terms of the stress invariant:
\begin{equation} \label{eq:volumetric_strain}
\tr \, \bfeps = \dfrac{(E_1 + d \, E_2) \, \tr \bfT}{ \left[  1 + (a_3 - E_1 a_1 - d E_2 a_2) \tr \bfT           \right]}
\end{equation}
In these equations, $d$ denotes the dimension of the problem.  By substituting the explicit expression for $\tr \, \bfeps $ from Eq. \eqref{eq:subclass_linear} back into the original constitutive framework, we arrive at a closed-form representation of the linearized strain:
\begin{equation} \label{eq:strain_full_sub}
\bfeps = E_1 \, \Phi_1(\tr \bfT) \, \bfT + E_2 \, \Phi_2(\tr \bfT) \, \tr \bfT \, \mathbf{I}
\end{equation}
To facilitate the analysis of mathematical properties, this somewhat cumbersome expression is condensed using two auxiliary linear-fractional functions, $\Phi_1$ and $\Phi_2$. These auxiliary functions encapsulate the nonlinear dependence on the pressure and are defined as:
\begin{subequations}
\begin{align}
\Phi_1(\tr \bfT) &:= \dfrac{ 1 + \left(a_3  + (a_1 -a_2) \, E_2 \, d \right) \tr \bfT}{(1 + a_3 \tr \bfT) \left[  1 + (a_3 - E_1 a_1 - d E_2 a_2) \tr \bfT           \right]} \\
& \notag \\
\Phi_2(\tr \bfT) &:= \dfrac{ 1 + \left(a_3  + (a_1 -a_2) \, E_1  \right) \tr \bfT}{(1 + a_3 \tr \bfT) \left[  1 + (a_3 - E_1 a_1 - d E_2 a_2) \tr \bfT           \right]} 
\end{align}
\end{subequations}
where the critical stress thresholds govern the singular behavior:
\begin{subequations}
\begin{align}
T_{cr1} &= \dfrac{-1}{a_3} \\
T_{cr2} &= \dfrac{-1}{a_3 - E_1 a_1 - d E_2 a_2} 
\end{align}
\end{subequations}
The functions $\Phi_1(\cdot)$ and $\Phi_2(\cdot)$ are neither bounded (from below or above) nor even continuous, so there is no well-posedness theory one can apply to the constitutive relation \eqref{eq:strain_full_sub}. Figure~\ref{fig:phi1} depitcts the function $\Phi_1(\cdot)$ for certain choices 
\begin{figure}[H]
    \centering
    \includegraphics[width=\textwidth]{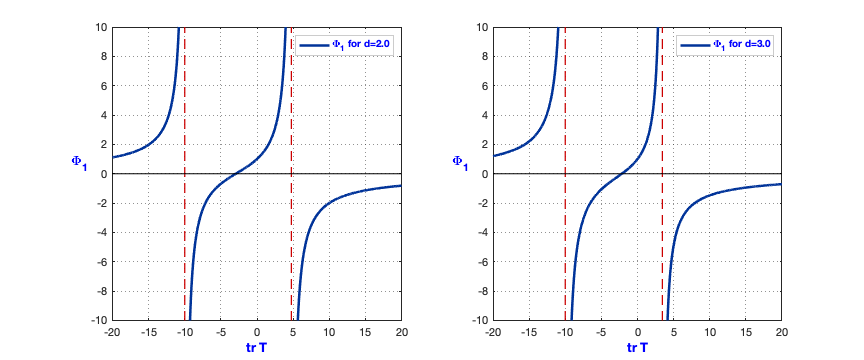} 
    \caption{Numerical evaluation of the response function $\Phi_1(\tr \, \bfT)$ for $d = 2.0$ (left) and $d = 3.0$ (right). The solid blue curves illustrate the function's behavior, while the dashed red lines indicate the vertical asymptotes. Constants used in the computations are: $a_1 = 0.5, a_2 = 0.2, a_3 = 0.1, E_1 = 0.3, E_2 = 0.4$.}
    \label{fig:phi1}
\end{figure}
Our approach to derive the model for elastic solids with pointwise variations in density is based on the approach in \cite{itou2022investigation,itou2021implicit}. For the wellposedness of the model \eqref{eq:strain_full_sub} we choose $a_1 = a_3=0$ equivalently $\Phi_1(\tr \, \bfT) =1$ then $\Phi_2(\tr \bfT) = 1 + \beta \, \tr \bfT$ with $\beta \in \mathbb{R}$, finally the constitutive relation reduces to 
\begin{equation} \label{eq:CEquaiton}
\bfeps = E_1 \, \, \bfT + E_2 \, (1 + \beta \, \tr \bfT) \, \tr \bfT \, \mathbf{I}, 
\end{equation}
The hyperelastic formulation of the above constitutive relation \eqref{eq:CEquaiton} is 
\begin{equation} \label{eq:T}
\bfT = \dfrac{1}{E_1} \, \bfeps - E_2 \, \psi(\nabla \cdot \bfu) \; \nabla \cdot \bfu \, \mathbf{I}, 
\end{equation}
with 
\begin{equation}
\psi( \delta ) := \dfrac{1 + \beta \, \delta }{E_1 + d \, E_2 \, ( 1 + \beta \, \delta)}
\end{equation}

\section{Boundary value problem and existence of solution}\label{BVP}
Investigating cracks and fractures in isotropic elastic materials with nonlinear, density-dependent properties is essential for numerous critical engineering applications. While classical linear elasticity provides a fundamental basis, many advanced natural and engineered materials exhibit mechanical responses that are intrinsically coupled to local density variations. These materials include functionally graded composites, porous geological formations, cellular foams, and biological tissues such as bone. In these contexts, the presence of cracks not only compromises structural integrity but can also interact with density gradients to precipitate catastrophic failures. Consequently, a deep understanding of crack behavior within such nonlinear media is imperative for designing safe and reliable structures. For instance, in aerospace applications where lightweight cellular composites are prevalent, predicting crack propagation through density-varying zones is crucial for ensuring aircraft safety. Similarly, in civil engineering, comprehending the deterioration of concrete---which often behaves as a heterogeneous, density-dependent continuum---is paramount for infrastructure longevity. Therefore, a rigorous investigation of crack-tip fields in isotropic materials with density-dependent moduli is warranted. To this end, we formulate a well-posed boundary value problem governed by a specific nonlinear constitutive law and further propose an $hp$-adaptive finite element method for solving the nonlinear problem.

\subsection{Problem formulation}
We consider the balance of linear momentum coupled with a nonlinear constitutive relationship. The stress tensor $\bfT$ is defined as a function of the infinitesimal strain tensor $\bfeps(\bfu)$ and the volumetric strain $\nabla \cdot \bfu$ (equivalently $\tr(\bfeps)$). The proposed constitutive model is given by:
\begin{equation} \label{eq:NewConstitutive}
\bfT(\bfeps) = \frac{1}{E_1} \, \bfeps - E_2 \, \frac{1 + \beta \, (\nabla \cdot \bfu) }{E_1 + d \, E_2 \, (1 + \beta \, \nabla \cdot \bfu)}\; (\nabla \cdot \bfu) \, \bfI,
\end{equation}
where $E_1, E_2, \beta,$ and $d$ are material constants. Combining this with the equilibrium equations, we obtain the following boundary value problem:
\begin{subequations}
\begin{align}
-\nabla \cdot \bfT(\bfeps(\bfu)) &= \bff, \quad \mbox{in} \quad \mathcal{D}, \label{pde_new} \\
\bfT(\bfeps(\bfu)) \, \bfn &= \bfg, \quad \mbox{on} \;\; \Gamma_N, \label{NCond_new}\\
\bfu &= \bfu_0, \quad \mbox{on} \;\; \Gamma_D,
\end{align}
\end{subequations}
where $\bff \in \left(L^{2}(\mathcal{D}) \right)^2$ is the body force term and $\bfg \in \left( H^{-1/2}(\Gamma_N) \right)^2$ is the traction boundary condition.

For the well-posedness of the above model, the following assumptions are made regarding the data and parameters:
\begin{itemize}
\item[A1:] The material parameters $E_1, E_2, d, \beta$ are real numbers. Furthermore, we assume the coefficients satisfy the ellipticity condition such that the tensor $\bfT$ remains positive definite for all admissible strains.
\item[A2:] In the pure Neumann case ($\Gamma_D = \emptyset$), the data $\bff$ and $\bfg$ satisfy the standard compatibility condition (equilibrium of external forces):
\begin{equation}
\int_{\mathcal{D}} \bff \, d\bfx + \int_{\partial \mathcal{D}} \bfg \, ds = \bfzero.
\end{equation}
\item[A3:] The Dirichlet data $\bfu_0$ is the trace of a function in $\left( H^{1}(\mathcal{D}) \right)^2$.
\end{itemize}

\subsection{Weak Formulation and Existence of Solutions}

We define the space of admissible test functions
\begin{equation}
 \mathbb{V} = \{ \bfw \in (H^1(\mathcal{D}))^2 : \bfw|_{\Gamma_D} = \bfzero \}. 
 \end{equation}
 The weak formulation of the problem is obtained by multiplying \eqref{pde_new} by a test function $\bfw \in \mathbb{V}$ and integrating by parts:
 \begin{cwf}
Find $\bfu \in (H^1(\mathcal{D}))^2$ with $\bfu - \bfu_0 \in \mathbb{V}$ such that:
\begin{equation}\label{weak_form_new}
\int_{\mathcal{D}} \bfT(\bfeps(\bfu)) \cdot \bfeps(\bfw) \, d\bfx = \int_{\mathcal{D}} \bff \cdot \bfw \, d\bfx + \int_{\Gamma_N} \bfg \cdot \bfw \, ds, \quad \forall \bfw \in \mathbb{V}.
\end{equation}
\end{cwf}

For the existence of a unique solution to the continuous weak formulation, we have the following theorem.

\begin{Theorem}
Let $\mathcal{D} \subset \mathbb{R}^2$ be a bounded Lipschitz domain. Under assumptions A1-A3, and assuming the data is sufficiently small such that the system remains in the elliptic regime defined in Remark \ref{rem:singularity}, there exists a unique weak solution $\bfu \in (H^1(\mathcal{D}))^2$ to the problem.
\end{Theorem}

\begin{remark}\label{rem:singularity}
\textbf{Condition on parameters ($E_2 < 0, \beta < 0$):}
Since $E_1 > 0$ and $E_2 < 0$, the denominator in the constitutive law \eqref{eq:NewConstitutive} introduces a singularity. Expanding the denominator yields $E_1 + d E_2 + d E_2 \beta (\nabla \cdot \bfu)$. The singularity occurs at the critical volumetric strain:
\begin{equation}
\xi_{crit} = - \frac{E_1 + d E_2}{d E_2 \beta}.
\end{equation}
For the existence proof to be valid, the solution must satisfy the constraint $\nabla \cdot \bfu \neq \xi_{crit}$ (typically $\nabla \cdot \bfu < \xi_{crit}$ if approximating from zero) almost everywhere. Furthermore, to ensure strong ellipticity, the parameters must satisfy the condition that the local bulk modulus remains positive within the admissible range of strains.
\end{remark}

\begin{Proposition}\label{prop:monotonicity}
Let $\mathcal{A}(\bfu)$ be the operator defined by the weak form. The operator is strictly monotone provided that the volumetric strain $\xi = \nabla \cdot \bfu$ satisfies the bound defined in Remark \ref{rem:singularity} and the derivative of the scalar stress response function $h(\xi)$ satisfies:
\begin{equation}
\frac{1}{E_1} + h'(\xi) > 0, \quad \text{where} \quad h(\xi) = - E_2 \, \frac{\xi + \beta \xi^2 }{E_1 + d E_2 + d E_2 \beta \xi}.
\end{equation}
\end{Proposition}

\begin{proof}
(Sketch) The operator is composed of a linear part $\frac{1}{E_1}\bfeps$ and a nonlinear volumetric part. The linear part is strictly coercive since $E_1 > 0$. The potential for loss of existence arises from the nonlinear term $h(\xi)$.
Given $E_2 < 0$, the term acts as a stiffening or softening agent depending on the magnitude of $\xi$.
\begin{itemize}
\item \textit{Boundedness:} We restrict the domain of admissible functions to an open set $\mathbb{K} \subset H^1(\mathcal{D})$ such that the denominator $E_1 + d E_2 (1 + \beta \xi) \neq 0$.
\item \textit{Monotonicity:} We calculate the derivative of the volumetric stress contribution. Existence is guaranteed by the Lax-Milgram lemma (or its nonlinear generalization) only if the total stiffness remains positive definite.
\end{itemize}
Specifically, if $E_2$ and $\beta$ are negative, we require that the negative contribution from the nonlinearity does not exceed the positive coercivity of the linear term $\frac{1}{E_1} \bfeps$. Under the assumption of small strains, we have $\nabla \cdot \bfu \approx 0$, and the response is governed by the linearized modulus. Since the singularity is avoided for the small strain regime, the local monotonicity holds, and a unique solution exists in the neighborhood of zero.
\end{proof}

\section{Finite element discretization}
\label{sec:fem}

This section is dedicated to the numerical approximation of the nonlinear boundary value problem established in the previous section. Due to the nonlinear dependence of the stress tensor on the volumetric strain, obtaining a closed-form analytical solution is generally intractable. Consequently, we employ a conforming Galerkin finite element method to discretize the problem, transforming the continuous variational problem into a system of nonlinear algebraic equations. To solve this system efficiently, we formulate a linearization strategy based on the Newton-Raphson method.

\subsection{Discrete formulation}

We discretize the computational domain $\mathcal{D}$ using a family of shape-regular partitions $\{\mathcal{T}_h\}_{h > 0}$, consisting of non-overlapping quadrilateral finite elements $\mathcal{K}$. The mesh parameter is defined as $h := \max_{\mathcal{K} \in \mathcal{T}_h} \operatorname{diam}(\mathcal{K})$. We approximate the solution using the finite-dimensional subspace $\mathbb{V}_h \subset (H^1(\mathcal{D}))^2$ consisting of continuous, piecewise bilinear vector-valued functions:
\begin{equation}
\mathbb{V}_{h} := \left\{ \bfa{v}_h \in (C^{0}(\overline{\mathcal{D}}))^2 \colon \left. \bfa{v}_h \right|_{\mathcal{K}} \in (\mathbb{Q}_{1}(\mathcal{K}))^2, \forall \mathcal{K} \in \mathcal{T}_h \right\},
\end{equation}
where $\mathbb{Q}_{1}(\mathcal{K})$ denotes the space of bilinear polynomials on the element $\mathcal{K}$. We further define the subspace of test functions
\begin{equation}
\mathbb{V}_{h,0} = \{ \bfa{v}_h \in \mathbb{V}_h : \bfa{v}_h|_{\Gamma_D} = \bfzero \}.
\end{equation}

The discrete weak formulation seeks a displacement field $\bfa{u}_h \in \mathbb{V}_h$ such that the residual of the balance of momentum vanishes for all admissible test functions.

\begin{dwf}
Find $\bfa{u}_h \in \mathbb{V}_{h}$ (satisfying Dirichlet conditions on $\Gamma_D$) such that:
\begin{equation}\label{eq:discrete_residual}
\mathcal{R}(\bfa{u}_h; \bfa{v}_h) := a(\bfa{u}_h; \bfa{v}_h) - L(\bfa{v}_h) = 0, \quad \forall\, \bfa{v}_h \in \mathbb{V}_{h,0}.
\end{equation}
Here, the nonlinear semilinear form $a(\cdot; \cdot)$ and the linear functional $L(\cdot)$ are defined as:
\begin{align}
a(\bfa{u}_h; \bfa{v}_h) &:= \int_{\mathcal{D}} \bfT(\bfeps(\bfa{u}_h)) \cdot \bfeps(\bfa{v}_h) \, d\bfa{x}, \\
L(\bfa{v}_h) &:= \int_{\mathcal{D}} \bfa{f} \cdot \bfa{v}_h \, d\bfa{x} + \int_{\Gamma_N} \bfa{g} \cdot \bfa{v}_h \, ds,
\end{align}
where the stress tensor $\bfT$ is given by the constitutive law \eqref{eq:NewConstitutive}.
\end{dwf}

\subsection{Linearization and Newton's Method}

Since the form $\mathcal{R}(\bfa{u}_h; \bfa{v}_h)$ is nonlinear with respect to $\bfa{u}_h$, we employ the Newton-Raphson method to solve \eqref{eq:discrete_residual} iteratively. This requires the linearization of the residual equation. Let $\bfa{u}_h^k$ be the approximation at iteration $k$. We seek an increment $\delta \bfa{u}_h \in \mathbb{V}_{h,0}$ such that $\bfa{u}_h^{k+1} = \bfa{u}_h^k + \alpha^k \,  \delta \bfa{u}_h$ satisfies $\mathcal{R}(\bfa{u}_h^{k+1}; \bfa{v}_h) \approx 0$. By performing a Taylor series expansion and neglecting higher-order terms, we obtain the linearized problem:
\begin{equation}
\mathcal{R}(\bfa{u}_h^k; \bfa{v}_h) + D\mathcal{R}(\bfa{u}_h^k)[\delta \bfa{u}_h, \bfa{v}_h] = 0,
\end{equation}
where $D\mathcal{R}(\cdot)[\cdot, \cdot]$ is the G\^{a}teaux derivative.

To derive the linearized form, we differentiate the stress tensor $\bfT$ with respect to the displacement field. Recall the constitutive relation:
\begin{equation}
\bfT(\bfeps) = \frac{1}{E_1} \bfeps - E_2 \; f(\nabla \cdot \bfu) \; \nabla \cdot \bfu \; \mathbf{I} 
\end{equation}
 The scalar volumetric function is:
\begin{equation}
f(\xi) = \frac{1 + \beta \xi}{E_1 + d E_2 + d E_2 \beta \xi}.
\end{equation}
The directional derivative of the stress tensor in the direction of the increment $\delta \bfa{u}$ is:
\begin{equation}
\frac{\partial \bfT}{\partial \bfeps} : \bfeps(\delta \bfa{u}) = \frac{1}{E_1} \bfeps(\delta \bfa{u}) - E_2 \, \left[ f^{\prime}(\nabla \cdot \delta \bfa{u}) \, \nabla \cdot \bfa{u} +  \, f(\nabla \cdot \bfa{u}) \, \nabla \cdot \delta \bfa{u}  \right] \mathbf{I}.
\end{equation}
Consequently, the tangent bilinear form (Jacobian) is given by:
\begin{align}
\mathcal{J}(\bfa{u}_h^k; \delta \bfa{u}_h, \bfa{v}_h) &:= D\mathcal{R}(\bfa{u}_h^k)[\delta \bfa{u}_h, \bfa{v}_h] \notag \\
&= \int_{\mathcal{D}} \left[ \frac{1}{E_1} \bfeps(\delta \bfa{u}_h) \cdot \bfeps(\bfa{v}_h) - E_2 \,  \left[ f^{\prime}(\nabla \cdot \delta \bfa{u}) \, \nabla \cdot \bfa{u}^k +  \, f(\nabla \cdot \bfa{u}^k) \, \nabla \cdot \delta \bfa{u}  \right] \nabla \cdot \bfa{v}  \right] \, d\bfa{x}.
\end{align}

To effectively solve the nonlinear boundary value problem formulated above, we employ an iterative numerical strategy. The complexity of the constitutive relationship requires a robust linearization technique to resolve the dependence of the stress tensor on the deformation field. Consequently, we utilize a \textit{Newton-Raphson scheme augmented with a line search algorithm}. This approach ensures both quadratic convergence near the solution and global stability during the early iterations, particularly given the potential singularities in the material model. The detailed algorithmic steps for obtaining the discrete solution are outlined below.

\begin{Algorithm}[Newton-Raphson scheme with line search]
Initialize $\bfa{u}_h^0$ (e.g., solution to the linear problem with $\beta=0$). Set tolerance $\epsilon_{tol}$. \\
For $k = 0, 1, 2, \dots$:
\begin{enumerate}
    \item Compute the residual vector $\mathbf{R}^k$ and stiffness matrix $\mathbf{K}^k$ by assembling element contributions:
    \begin{equation}
    K_{ij}^k = \mathcal{J}(\bfa{u}_h^k; \bfa{\phi}_j, \bfa{\phi}_i), \quad R_i^k = -\mathcal{R}(\bfa{u}_h^k; \bfa{\phi}_i),
    \end{equation}
    where $\{\bfa{\phi}_i\}$ are the basis functions of $\mathbb{V}_{h,0}$.
    
    \item Solve the linear system for the search direction $\delta \mathbf{U}^k$:
    \begin{equation}
    \mathbf{K}^k \, \delta \mathbf{U}^k = \mathbf{R}^k.
    \end{equation}
    
    \item \textbf{Line Search:} Determine the step length $\alpha^k \in (0, 1]$ to ensure sufficient decrease in the residual. initialize $\alpha^k = 1$.
    \begin{itemize}
        \item While $\| \mathcal{R}(\bfa{u}_h^k + \alpha^k \delta \bfa{u}_h) \| > (1 - \gamma \alpha^k) \| \mathcal{R}(\bfa{u}_h^k) \|$:
        \begin{equation*}
        \alpha^k \leftarrow \rho \, \alpha^k,
        \end{equation*}
        where $\gamma \in (0, 1)$ is a sufficient decrease parameter (e.g., $10^{-4}$) and $\rho \in (0, 1)$ is a reduction factor (e.g., $0.5$).
    \end{itemize}
    
    \item Update the solution: $\bfa{u}_h^{k+1} = \bfa{u}_h^k + \alpha^k \, \delta \bfa{u}_h$.
    
    \item Check convergence: If $\| \mathbf{R}^{k+1} \| < \epsilon_{tol} \| \mathbf{R}^0 \|$, stop.
\end{enumerate}
\end{Algorithm}
The use of the exact stiffness matrix $\mathbf{K}^k$ combined with a globalization strategy via line search ensures robust convergence, even when the initial guess is far from the solution or near the constitutive singularity.

\subsection{Computational framework: The $hp$-adaptive finite element method}

Efficiently resolving the multiscale features inherent in fracture mechanics—specifically, the sharp gradients near the crack tip contrasted with the smooth far-field behavior—requires a discretization strategy that transcends fixed-grid approaches. In this work, we employ the $hp$-adaptive finite element method. Unlike standard adaptive schemes that rely solely on mesh subdivision ($h$-adaptivity) or uniform polynomial enrichment ($p$-adaptivity), the $hp$-method dynamically optimizes both the element size, $h$, and the approximation order, $p$. This synergistic approach allows for exponential convergence rates (spectral accuracy) in smooth regions while robustly resolving singularities via geometric refinement, minimizing the total degrees of freedom required for a target accuracy level \cite{bangerth2013adaptive, demkowicz2006computing}.

\subsubsection{Mechanisms of refinement and spectral enrichment}
The $hp$-framework operates by arbitrating between two distinct modes of convergence based on the local regularity of the solution:

\begin{enumerate}
    \item \textit{Spatial discretization ($h$-refinement):} This mode involves the recursive subdivision of finite elements. It is the dominant strategy for capturing low-regularity features, such as the singular strain fields characteristic of the fracture process zone or geometric discontinuities at re-entrant corners \cite{verfurth1994posteriori}. While $h$-refinement offers algebraic convergence rates, it is indispensable for isolating non-smooth behaviors that cannot be efficiently resolved by high-order polynomials alone.
    
    \item \textit{Spectral enrichment ($p$-refinement):} In subdomains where the solution is analytic or globally smooth, increasing the polynomial degree $p$ of the shape functions yields significantly faster error decay compared to mesh refinement. This spectral convergence allows the solver to capture complex, smooth variations in the elastic field without the prohibitive cost of a dense mesh \cite{duster2001p, babuska1981p}.
\end{enumerate}

\subsubsection{The Kelly error estimator}
To drive the adaptive process, we require a quantitative measure of the local discretization error. We utilize a residual-based a posteriori error estimator, specifically the Kelly error indicator \cite{kelly1983posteriori}. This estimator approximates the error in the energy norm by evaluating the violation of continuity requirements across inter-element boundaries.

Mathematically, the exact solution $u$ must satisfy flux continuity across all internal interfaces. However, the discrete finite element solution, $u_h$, typically exhibits discontinuities in its gradient, $\nabla u_h$, across element faces. The Kelly estimator, $\eta_K$, quantifies this lack of conformity for a cell $K$ as:
\begin{equation}
\eta_K = \left( \frac{h_K}{24} \int_{\partial K} \left[ \jump{\nabla u_h \cdot \mathbf{n}} \right]^2 \, ds \right)^{1/2},
\end{equation}
where $h_K$ denotes the characteristic cell diameter, $\mathbf{n}$ is the unit normal vector, and $\jump{\cdot}$ represents the jump operator across the face $\partial K$. 

\paragraph{Physical interpretation in fracture mechanics:}
The term $\jump{\nabla u_h \cdot \mathbf{n}}$ represents the jump in the normal flux (or traction components in elasticity) across element boundaries. In the context of the strain-limiting fracture model, regions of high curvature—such as the crack tip vicinity—generate large flux discontinuities when under-resolved. The Kelly estimator naturally identifies these regions as areas of high error, ensuring that the adaptive algorithm concentrates computational resources precisely at the singular front, rather than wasting degrees of freedom in the far field.

\subsubsection{The $hp$-algorithm}
The efficacy of the $hp$-method hinges on a robust marking strategy that not only identifies \textit{where} to refine (based on the magnitude of $\eta_K$) but determines \textit{how} to refine (choosing between $h$ and $p$). The implemented strategy follows a four-stage cycle:\\
\begin{center}
 \textsc{Solve} $\to$ \textsc{Estimate} $\to$ \textsc{Mark} $\to$ \textsc{Refine}.
\end{center}
~\\
The decision to apply $h$- versus $p$-refinement is governed by a regularity estimator. This algorithm analyzes the decay rate of the Legendre expansion coefficients of the local solution. 
\begin{itemize}
    \item \textit{High regularity:} If the coefficients decay rapidly, the solution is deemed locally smooth. The algorithm marks the cell for $p$-enrichment, leveraging spectral convergence.
    \item \textit{Low regularity:} If the coefficient decay is slow (indicating a singularity or sharp gradient), the algorithm flags the cell for $h$-subdivision, isolating the singularity geometrically.
\end{itemize}
This dynamic selection ensures that the mesh topology and polynomial distribution evolve to match the underlying physics of the strain-limiting solid. We utilize the \texttt{deal.II} library \cite{arndt2021deal} to manage the complex data structures required for such multi-constraint, non-conforming meshes.

\section{Numerical results and discussion}\label{sec:results}
In this section, we substantiate the theoretical properties of the proposed isotropic model with density-dependent moduli through rigorous numerical experimentation. The focus is placed on characterizing the impact of the constitutive nonlinearity on the crack-tip fields. Specifically, we examine how the density-dependent stiffening (or softening) alters the stress distribution compared to the predictions of classical linear elastic fracture mechanics (LEFM). The governing equations for the static equilibrium of the solid represent a system of second-order quasilinear partial differential equations. To address the inherent nonlinearity arising from the volumetric strain dependence, we implement the damped Newton-Raphson iterative scheme (with line search) within an $hp$-adaptive framework. This requires the consistent derivation of the Fréchet derivative (the Jacobian) of the residual operator. At every Newton step, a linearized system is solved, and the mesh is dynamically adapted to capture the high gradients in the evolving fracture process zone.

We simulate a canonical mode-I, mode-II, and mixed mode crack problem to benchmark the solver. This setup is essential because the proposed nonlinearity is driven by volumetric strain ($\nabla \cdot \bfu$), which is dominant in opening-mode fractures. As demonstrated in the subsequent results, the $hp$-adaptive discretization successfully resolves the complex field gradients at the crack tip, verifying the stability of the algorithm even in the presence of constitutive singularities. The computational procedure follows the steps summarized in Algorithm \ref{alg:hp_adaptive_fem}. The entire computational code is developed using \textsf{deal.II} library. To illustrate the response of the proposed model, we consider the following loading scenarios.

\begin{figure}[H]
    \centering
    \tikzset{
        mainbody/.style={thick, draw=black}, 
        crackline/.style={line width=2pt, draw=red},
        forcearrow/.style={->, >=latex, thick, blue}
    }
    \begin{minipage}[b]{0.32\textwidth}
        \centering
        \begin{tikzpicture}[scale=3.5]
            \draw[mainbody] (0,0) rectangle (1,1);
            \draw[crackline] (1.0, 0.5) -- (0.5, 0.5);
            \fill[red] (0.5, 0.5) circle (0.5pt); 
            \foreach \x in {0,0.2,...,1.0} \draw (\x,0) -- (\x-0.05,-0.05);
            \node[below, font=\tiny] at (0.5, -0.05) {$\mathbf{u}=\mathbf{0}$};
            \foreach \x in {0,0.2,...,1.0} \draw[forcearrow] (\x, 1) -- (\x, 1.2);
            \node[blue, above, font=\tiny] at (0.5, 1.2) {$\mathbf{u}=(0, \bar{v})$};
            \node[font=\small\bfseries] at (0.5, -0.3) {(a) Tensile};
        \end{tikzpicture}
    \end{minipage}
    \hfill
    \begin{minipage}[b]{0.32\textwidth}
        \centering
        \begin{tikzpicture}[scale=3.5]
            \draw[mainbody] (0,0) rectangle (1,1);
            \draw[crackline] (1.0, 0.5) -- (0.5, 0.5);
            \fill[red] (0.5, 0.5) circle (0.5pt); 
            \foreach \x in {0,0.2,...,1.0} \draw (\x,0) -- (\x-0.05,-0.05);
            \node[below, font=\tiny] at (0.5, -0.05) {$\mathbf{u}=\mathbf{0}$};
            \foreach \x in {0,0.2,...,1.0} \draw[forcearrow] (\x, 1) -- (\x+0.2, 1);
            \node[blue, above, font=\tiny] at (0.5, 1.05) {$\mathbf{u}=(\bar{u}, 0)$};
            \node[font=\small\bfseries] at (0.5, -0.3) {(b) Shear};
        \end{tikzpicture}
    \end{minipage}
    \hfill
    \begin{minipage}[b]{0.32\textwidth}
        \centering
        \begin{tikzpicture}[scale=3.5]
            \draw[mainbody] (0,0) rectangle (1,1);
            \draw[crackline] (1.0, 0.5) -- (0.5, 0.5);
            \fill[red] (0.5, 0.5) circle (0.5pt); 
            \foreach \x in {0,0.2,...,1.0} \draw (\x,0) -- (\x-0.05,-0.05);
            \node[below, font=\tiny] at (0.5, -0.05) {$\mathbf{u}=\mathbf{0}$};
            \foreach \x in {0,0.2,...,1.0} \draw[forcearrow] (\x, 1) -- (\x+0.15, 1.2);
            \node[blue, above, font=\tiny] at (0.5, 1.2) {$\mathbf{u}=(\bar{u}, \bar{v})$};
            \node[font=\small\bfseries] at (0.5, -0.3) {(c) Mixed};
        \end{tikzpicture}
    \end{minipage}

    \caption{Schematic representation of the boundary value problems for (a) tensile, (b) shear, and (c) mixed-mode loading.}
    \label{fig:loading_scenarios}
\end{figure}
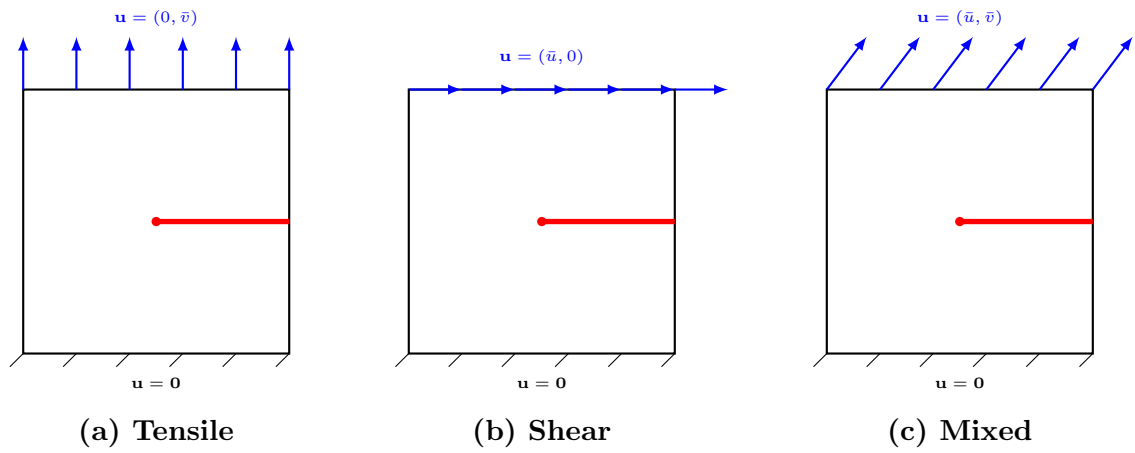

\begin{algorithm}[H]
    \caption{Staggered $hp$-Adaptive Newton-Raphson Scheme}
    \label{alg:hp_adaptive_fem}
    \begin{algorithmic}[1]
        \State \textbf{Input:} Initial coarse mesh $\mathcal{T}_0$, initial polynomial distribution $\mathbb{Q}_0$, tolerance $\epsilon_{\text{tol}}$.
        \State \textbf{Initialize:} Adaptive step $k \leftarrow 0$; Initial solution guess $\bfa{u}_{h,0}$.
        
        \Repeat \Comment{\textbf{Outer Adaptive Loop}}
            \State Initialize Newton iteration $n \leftarrow 0$; Set iterate $\bfa{u}^{(0)} \leftarrow \bfa{u}_{h,k}$.
            
            \Repeat \Comment{\textbf{Inner Nonlinear Solver}}
                \State Assemble global Tangent Stiffness $\mathbf{K}^{(n)}$ and Residual vector $\mathbf{R}^{(n)}$:
                $$ \mathbf{K}^{(n)}_{ij} = \mathcal{J}(\bfa{u}^{(n)}; \bfa{\phi}_j, \bfa{\phi}_i), \quad \mathbf{R}^{(n)}_i = -\mathcal{R}(\bfa{u}^{(n)}; \bfa{\phi}_i) $$
                
                \State Solve linear system for the increment $\delta \bfa{u}^{(n)}$:
                $$ \mathbf{K}^{(n)} \delta \bfa{u}^{(n)} = \mathbf{R}^{(n)} $$
                
                \State \textbf{Line Search:} Find step length $\alpha^{(n)} \in (0,1]$ via backtracking:
                \State \quad While $\| \mathcal{R}(\bfa{u}^{(n)} + \alpha^{(n)} \delta \bfa{u}^{(n)}) \| > (1 - \gamma \alpha^{(n)}) \| \mathbf{R}^{(n)} \|$, set $\alpha^{(n)} \leftarrow \rho \alpha^{(n)}$.
                
                \State Update solution: $\bfa{u}^{(n+1)} \leftarrow \bfa{u}^{(n)} + \alpha^{(n)} \delta \bfa{u}^{(n)}$.
                \State $n \leftarrow n + 1$.
            \Until{$\| \mathbf{R}^{(n)} \| < \epsilon_{\text{Newton}} \| \mathbf{R}^{(0)} \|$}
            
            \State Set converged solution on current mesh: $\bfa{u}_{h,k} \leftarrow \bfa{u}^{(n)}$.
            
            \State \textbf{Error Estimation:}
            \State \quad Compute local error indicators $\eta_K(\bfa{u}_{h,k})$ for all $K \in \mathcal{T}_k$.
            \State \quad Compute global error estimate $\mathcal{E}_{global} = \left(\sum_{K \in \mathcal{T}_k} \eta_K^2\right)^{1/2}$.
            
            \If{$\mathcal{E}_{global} < \epsilon_{\text{tol}}$}
                \State \textbf{Converged:} Break adaptive loop.
            \EndIf
            
            \State \textbf{Marking \& Refinement:}
            \State \quad Identify elements for refinement: $\mathcal{M}_h$ (geometric) and $\mathcal{M}_p$ (polynomial) based on smoothness indicators.
            \State \quad Generate new mesh: $\mathcal{T}_{k+1} \leftarrow \text{Refine}(\mathcal{T}_k, \mathcal{M}_h)$.
            \State \quad Update ansatz space: $\mathbb{Q}_{k+1} \leftarrow \text{Enrich}(\mathbb{Q}_k, \mathcal{M}_p)$.
            \State \quad \textbf{Transfer:} Project $\bfa{u}_{h,k}$ to the new mesh $\mathcal{T}_{k+1}$ to serve as initial guess $\bfa{u}_{h, k+1}$.
            \State $k \leftarrow k + 1$.
            
        \Until{Max DOFs or iterations reached}
        \State \textbf{Output:} Final displacement field $\bfa{u}_{h,k}$.
    \end{algorithmic}
\end{algorithm}

\subsection{Mode I: Uniaxial tensile loading}
We investigate the behavior of a static crack within an elastic solid governed by the proposed nonlinear constitutive law, specifically under Mode I tensile loading. The resulting system of vector-valued nonlinear partial differential equations is approximated using the $hp$-adaptive finite element framework detailed in Algorithm~\ref{alg:hp_adaptive_fem}. The computational domain and the corresponding boundary conditions are illustrated in Figure~\ref{fig:loading_scenarios}(a). The adaptive process was executed for a minimum of 10 refinement cycles, allowing the basis functions to reach a maximum polynomial degree of $p=6$. Figure~\ref{c_plot_tensile} depicts the convergence history of the adaptive solution. We observe an asymptotic convergence rate of approximately $0.64$; considering the strong singularity inherent to the nonlinear crack-tip field, this rate demonstrates the robustness and acceptability of the numerical scheme.

\begin{figure}[H]
\centering
\includegraphics[width=0.6\textwidth]{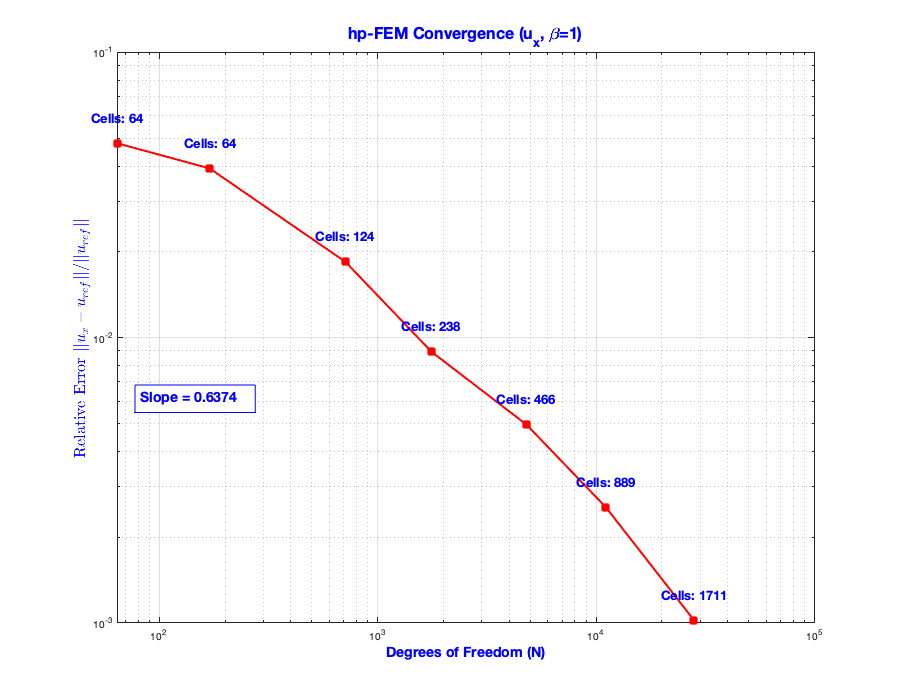} 
\caption{Convergence plot for the case of tensile load. }
\label{c_plot_tensile}
\end{figure}

\begin{figure}[H]
    \centering
    \begin{subfigure}[b]{0.48\textwidth}
        \centering
        \includegraphics[width=\textwidth]{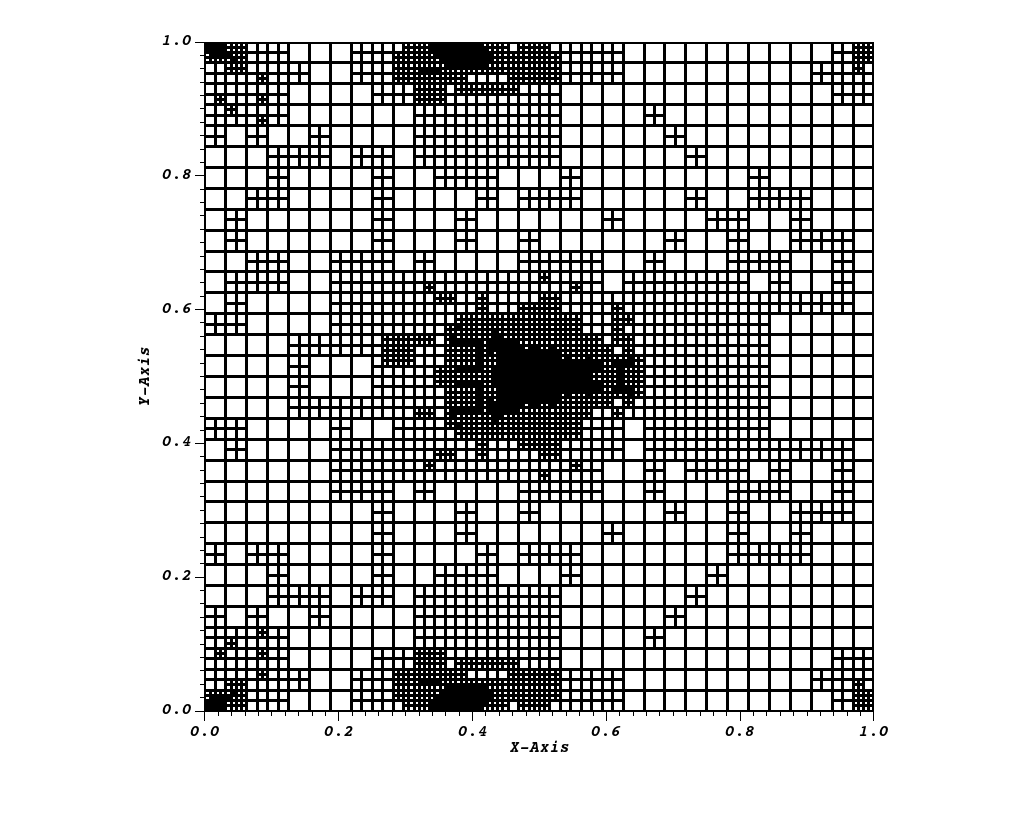}
        \caption{Mesh}
        \label{tensile_mesh}
    \end{subfigure}
    \hfill 
    \begin{subfigure}[b]{0.48\textwidth}
        \centering
        \includegraphics[width=\textwidth]{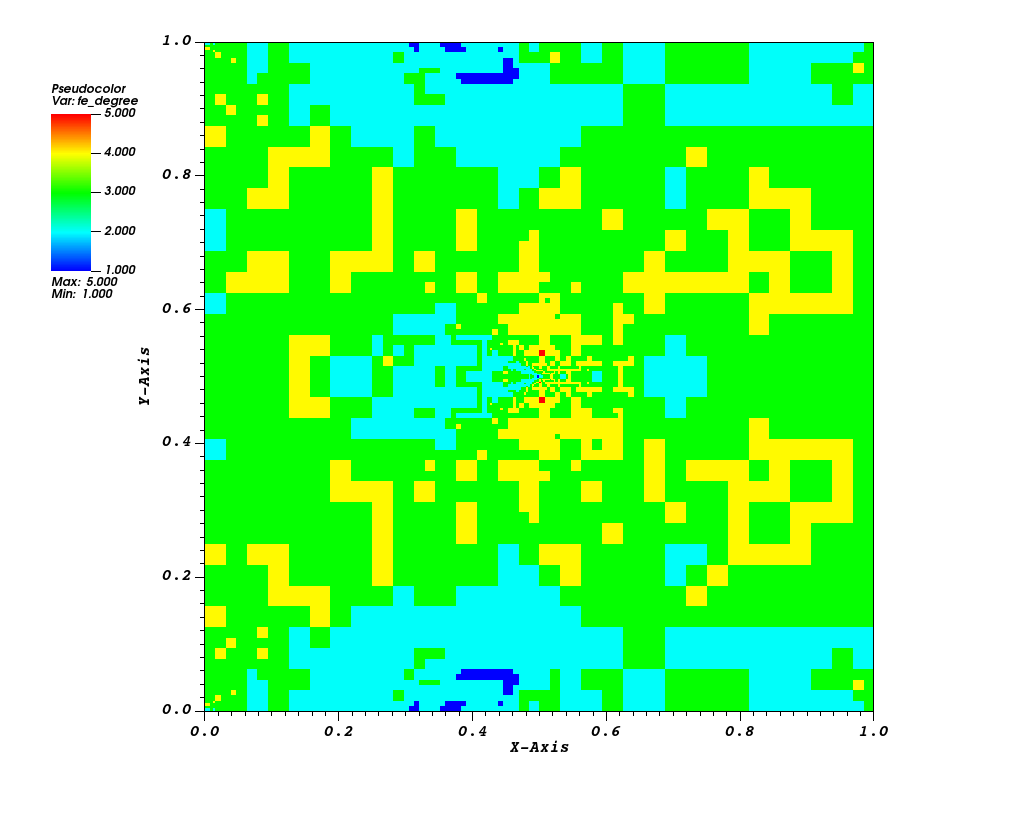}
        \caption{Polymomial degree}
        \label{fedeg_tenisle}
    \end{subfigure}  
    \caption{The state of the discretization after 8 adaptive refinement steps under tensile loading. (a) shows the geometrically graded mesh concentrating on the crack tip, while (b) shows the corresponding distribution of polynomial orders ranging from $p=1$ (blue) to $p=5$ (orange). }
    \label{fig:hp_mesh_results}
\end{figure}

The efficacy of the proposed $hp$-adaptive refinement strategy is visually demonstrated in Figure \ref{fig:hp_mesh_results}, which illustrates the discrete state of the domain after 8 adaptive iterations under tensile loading. Figure \ref{fig:hp_mesh_results}(a) depicts the final computational mesh, revealing a highly localized geometric ($h$-) refinement focused intensely around the crack tip at $(0.5, 0.5)$. This aggressive subdivision is driven by the error estimator to resolve the high stress gradients and the singularity characteristic of the nonlinear fracture mechanics problem. Complementing the mesh topology, Figure \ref{fig:hp_mesh_results}(b) presents the spatial distribution of the polynomial degrees ($p$-levels). The algorithm autonomously identifies the low-regularity region near the crack tip, assigning lower-order basis functions ($p=1$, indicated in blue) to the smallest elements to maintain stability. Conversely, in the far-field regions where the deformation field is smooth, the method enriches the approximation space with higher-order polynomials (up to $p=5$, indicated in yellow/orange). This dual strategy of creating a graded mesh with variable polynomial orders confirms that the solver is effectively optimizing the degrees of freedom to capture the complex physics of the density-dependent constitutive model.

\begin{figure}[H]
    \centering
    \begin{subfigure}[b]{0.48\textwidth}
        \centering
        \includegraphics[width=\textwidth]{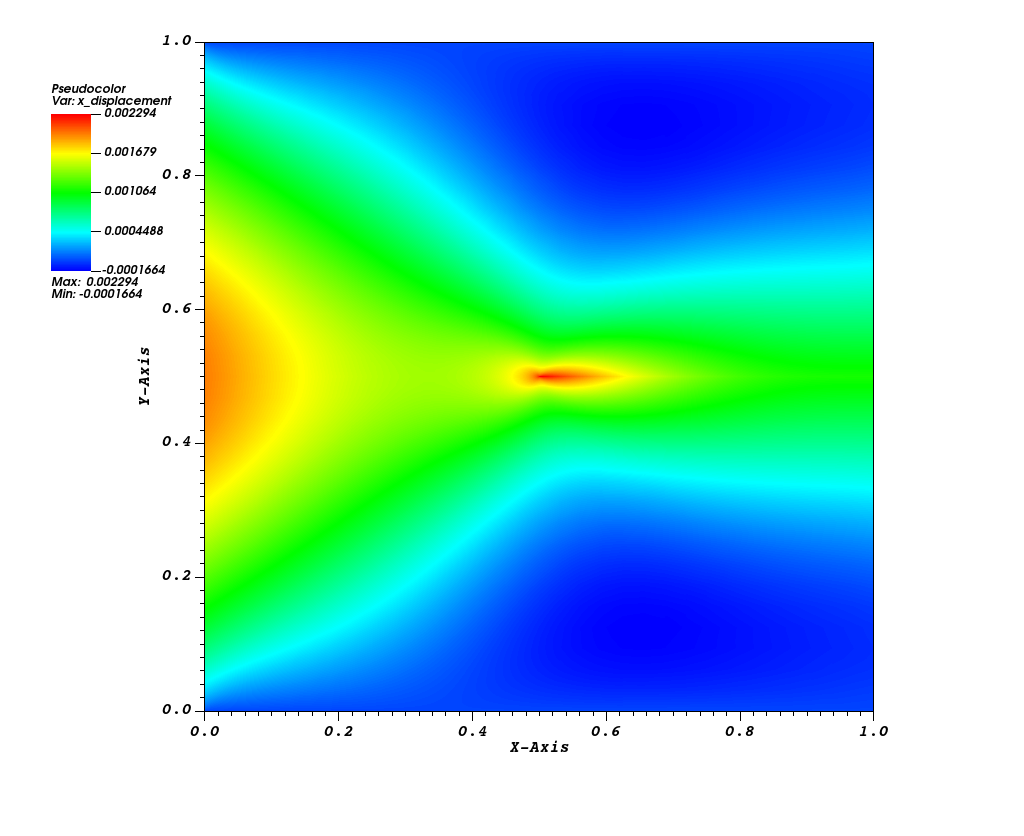}
        \caption{$\bfu_x$}
        \label{tensile_ux}
    \end{subfigure}
    \hfill 
    \begin{subfigure}[b]{0.48\textwidth}
        \centering
        \includegraphics[width=\textwidth]{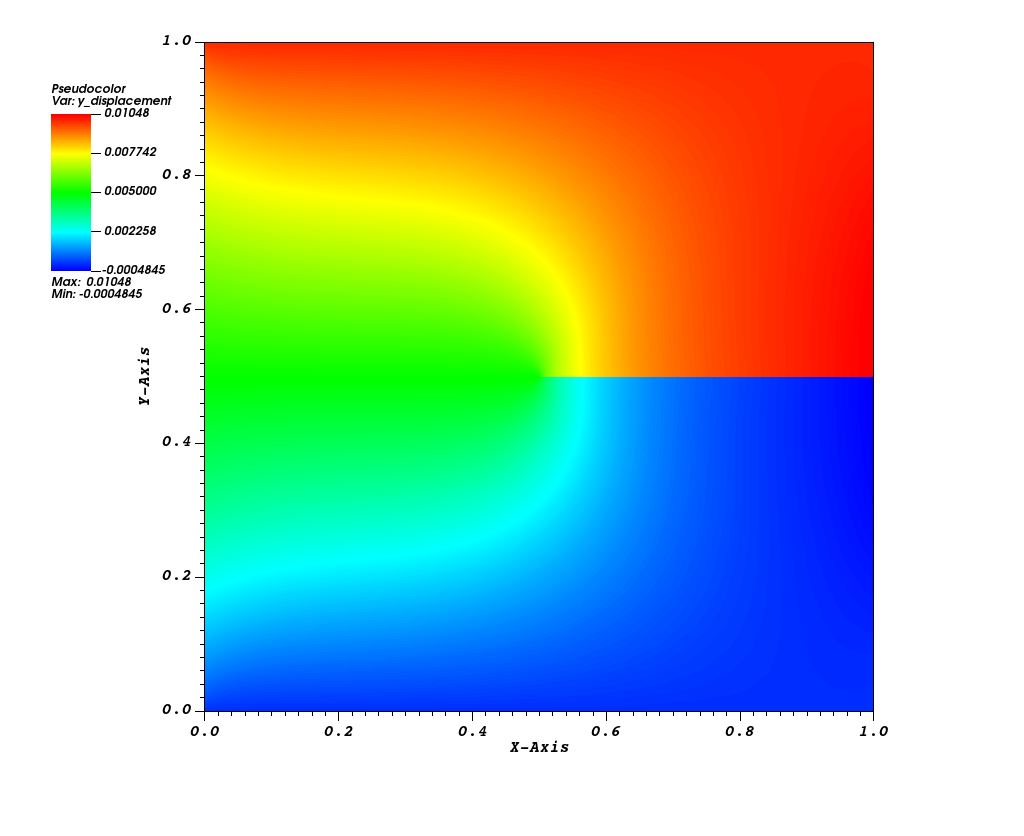}
        \caption{$\bfu_y$}
        \label{tensile_uy}
    \end{subfigure}  
\caption{Computed displacement fields for the tensile test problem. (a) shows the lateral contraction and tip deformation in the horizontal direction, while (b) clearly illustrates the crack opening discontinuity along the interface $y=0.5$.}
    \label{ux_uy_tensile}
\end{figure}

The computed displacement fields, obtained after 8 adaptive refinement steps, are visualized in Figure \ref{ux_uy_tensile}. Figure \ref{ux_uy_tensile}(b) displays the vertical displacement component ($\bfu_y$), which captures the characteristic kinematics of Mode I fracture. A sharp discontinuity is evident along the branch cut ($x \in [0.5, 1.0], y=0.5$), representing the physical crack opening where the upper flank lifts under the applied tensile load while the lower flank remains constrained. This discontinuity resolves smoothly into a continuous field ahead of the crack tip at $(0.5, 0.5)$, confirming that the adaptive mesh successfully captures the geometry. Complementing this, Figure \ref{ux_uy_tensile}(a) presents the horizontal displacement ($\bfu_x$). It highlights the lateral Poisson contraction associated with the uniaxial extension, as well as the localized deformation gradients in the immediate vicinity of the crack tip. The absence of spurious oscillations near the singularity in both fields attests to the stability of the $hp$-adaptive scheme under the proposed nonlinear constitutive law.

\begin{figure}[H]
\centering
\includegraphics[width=\textwidth]{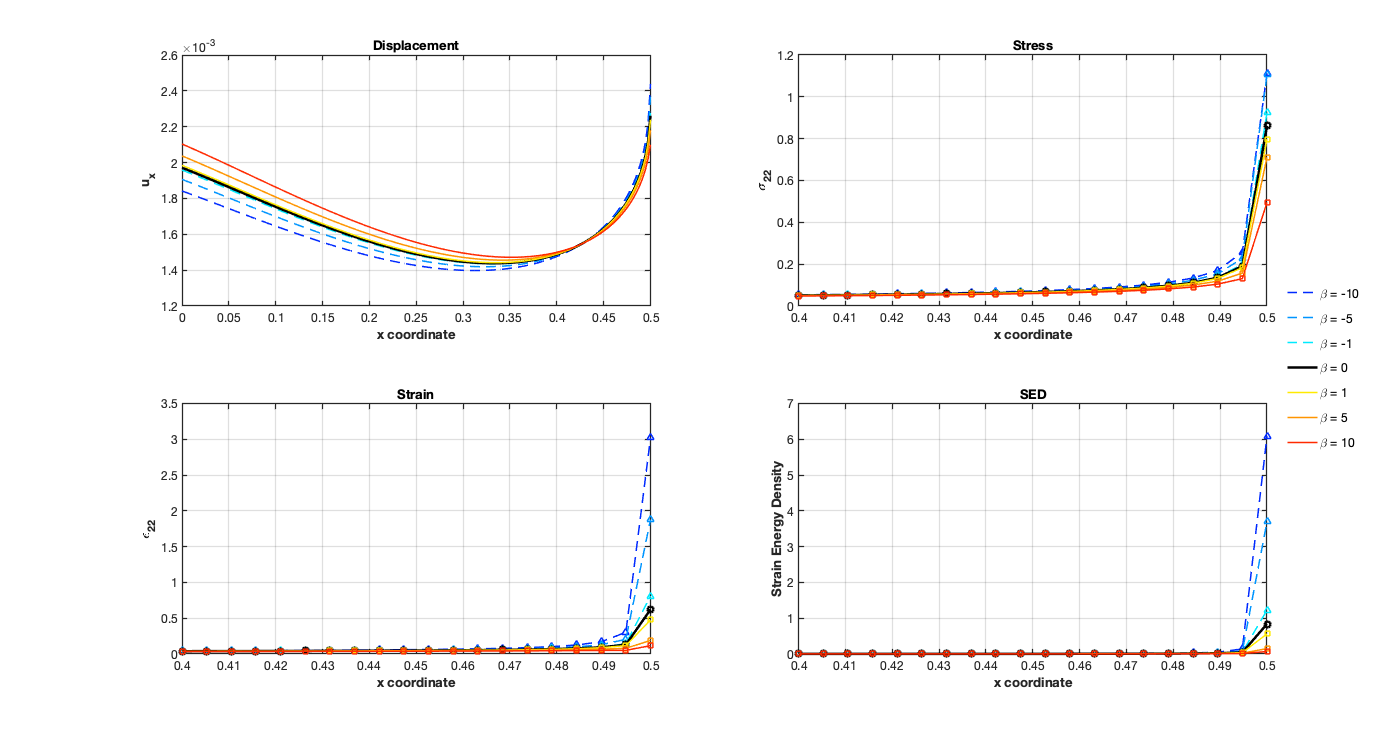} 
\caption{Distribution of field variables along the crack ligament ($y=0.5$) approaching the crack tip at $x=0.5$. The panels display (a) horizontal displacement $\bfu_x$, (b) vertical stress $\bfT_{22}$, (c) vertical strain $\bfeps_{22}$, and (d) Strain Energy Density (SED). The results highlight the contrast between the singularity-enhancing effects of negative $\beta$ (dashed lines) and the stress-relieving effects of positive $\beta$ (solid lines) compared to the linear elastic case ($\beta=0$, black line). }
\label{beta_plot_tensile}
\end{figure}

The numerical results presented in Figure \ref{beta_plot_tensile} elucidate the critical influence of the nonlinearity parameter $\beta$ on the near-tip fields. The plots compare the evolution of displacement ($\bfu_x$), vertical stress ($\bfT_{22}$), vertical strain ($\bfeps_{22}$), and SED along the ligament leading to the crack tip ($x \to 0.5$). A distinct dichotomy is observed between the negative and positive regimes of $\beta$. For negative values (e.g., $\beta = -10$, represented by dashed blue lines), the material exhibits a marked amplification of the crack-tip singularity. As predicted by the theoretical constraints, the interaction between a negative $\beta$ and the material constants pushes the constitutive denominator toward a critical value, resulting in a sharp elevation of both stress and strain well beyond the linear elastic baseline ($\beta=0$). Notably, the strain concentration ($\bfeps_{22}$) is disproportionately higher than the stress concentration, leading to a substantial accumulation of strain energy density. Conversely, positive values of $\beta$ (solid orange/red lines) appear to induce a screening or regularizing effect, effectively attenuating the singularity and resulting in field intensities lower than those predicted by linear elasticity. This parametric study confirms that $\beta$ acts as a tunable control parameter, capable of transitioning the material response from a highly singular, brittle-like state to a more compliant, energy-dissipating state.

\subsection{Mode II: Pure shear loading}
The discretization profile resulting from the $hp$-adaptive algorithm under pure shear (Mode II) loading is illustrated in Figure \ref{m_fedeg_shear}. Figure \ref{m_fedeg_shear}(a) displays the final $h$-refined mesh. Similar to the tensile case, the refinement is heavily concentrated around the crack tip at $(0.5, 0.5)$ to resolve the primary singularity. However, distinctive refinement patterns are also visible at the domain corners; these arise due to the stress concentrations induced by the shear displacement constraints interacting with the clamped boundaries. Figure \ref{m_fedeg_shear}(b) maps the corresponding polynomial degree distribution. The adaptive strategy effectively lowers the approximation order to $p=1$ (blue/cyan) in the singular zones to maintain stability and accuracy, while elevating the order to $p=4$ (red) in the smoother far-field regions. This distribution verifies that the error estimator correctly identifies the changing regularity of the solution driven by the shear-induced asymmetry.

\begin{figure}[H]
    \centering
    \begin{subfigure}[b]{0.48\textwidth}
        \centering
        \includegraphics[width=\textwidth]{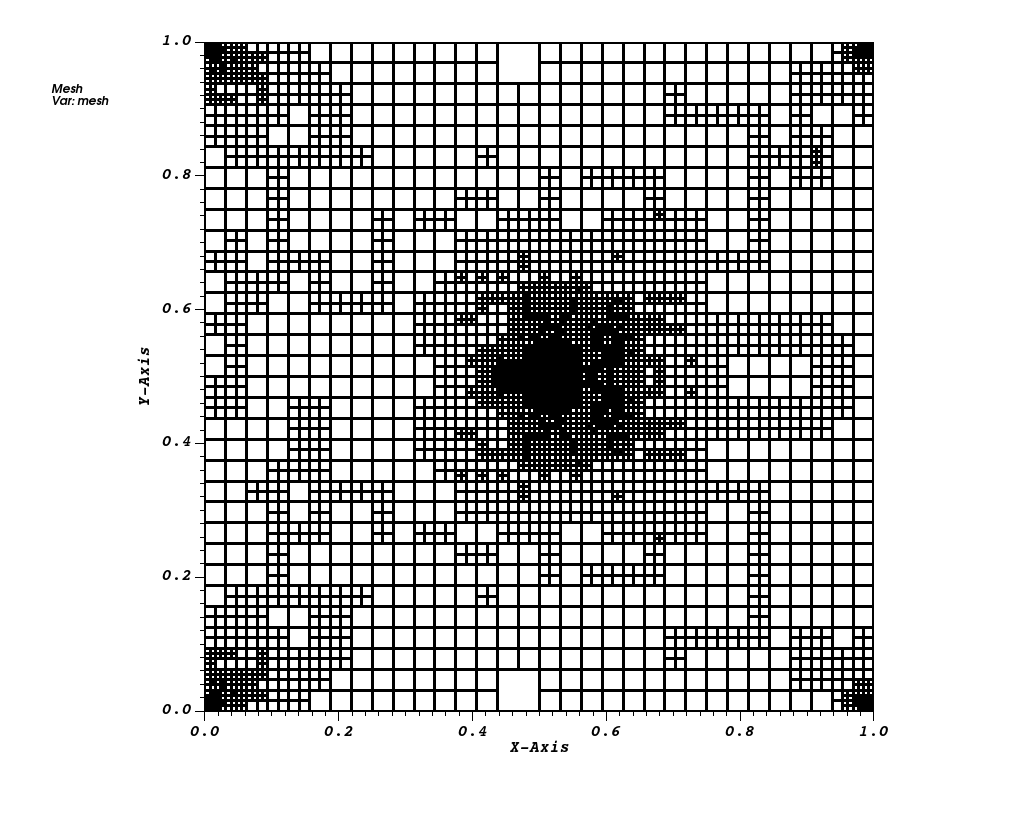}
        \caption{Mesh}
        \label{tensile_mesh}
    \end{subfigure}
    \hfill 
    \begin{subfigure}[b]{0.48\textwidth}
        \centering
        \includegraphics[width=\textwidth]{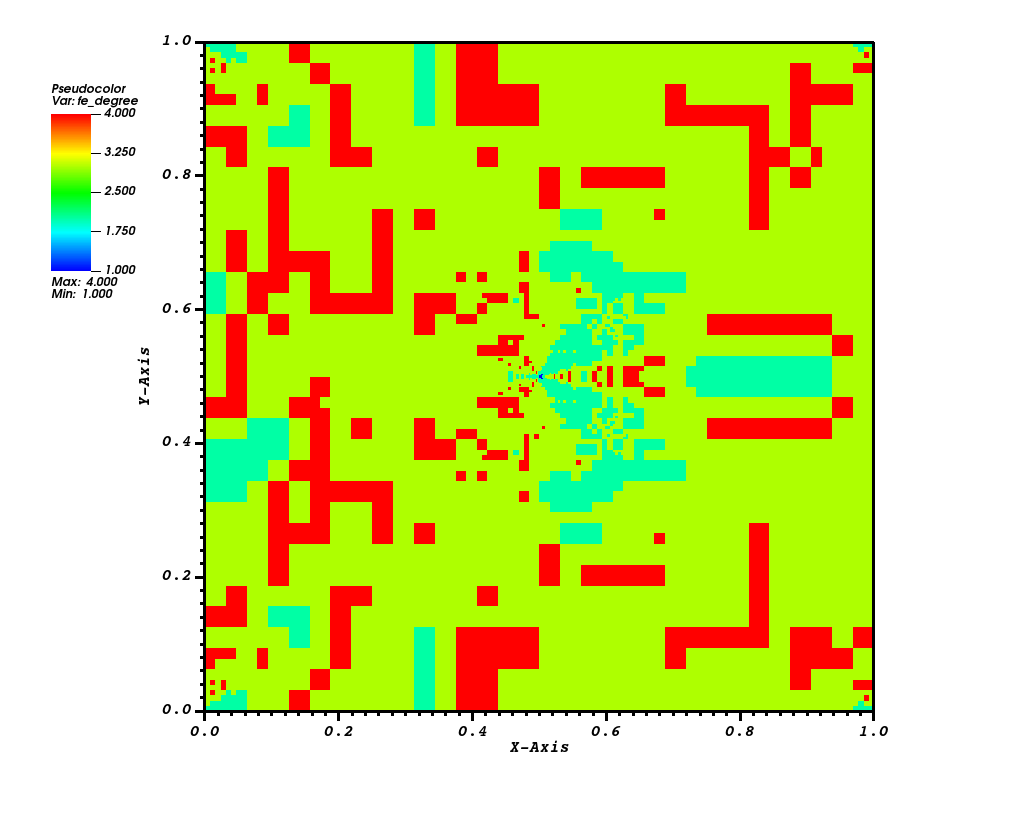}
        \caption{Polymomial degree}
        \label{fedeg_tenisle}
    \end{subfigure}  
    \caption{Adaptive discretization results for the shear loading (Mode II) case. (a) shows the mesh refinement concentrated at the crack tip and domain corners, while (b) illustrates the gradation of polynomial degrees, with lower orders near the singularity and higher orders in the smooth bulk regions.}
    \label{m_fedeg_shear}
\end{figure}

The computed displacement components for the Mode II shear problem are presented in Figure \ref{ux_uy_shear}. Figure \ref{ux_uy_shear}(a) illustrates the horizontal displacement $\bfu_x$, which characterizes the primary kinematic mechanism of shear fracture. A sharp discontinuity is evident along the crack faces ($x \in [0.5, 1.0], y=0.5$), representing the relative sliding of the upper and lower flanks in opposite directions. This jump in $\bfu_x$ confirms that the adaptive solver correctly captures the sliding mode without enforcing artificial continuity across the branch cut. In contrast, Figure \ref{ux_uy_shear}(b) depicts the vertical displacement $\bfu_y$. Unlike the tensile case, the vertical field remains continuous across the crack interface, indicating that there is no significant crack opening (Mode I contribution). Instead, the contours reveal the rotational deformation pattern and the anti-symmetric distribution of vertical material movement required to accommodate the shear distortion around the tip.

\begin{figure}[H]
    \centering
    \begin{subfigure}[b]{0.48\textwidth}
        \centering
        \includegraphics[width=\textwidth]{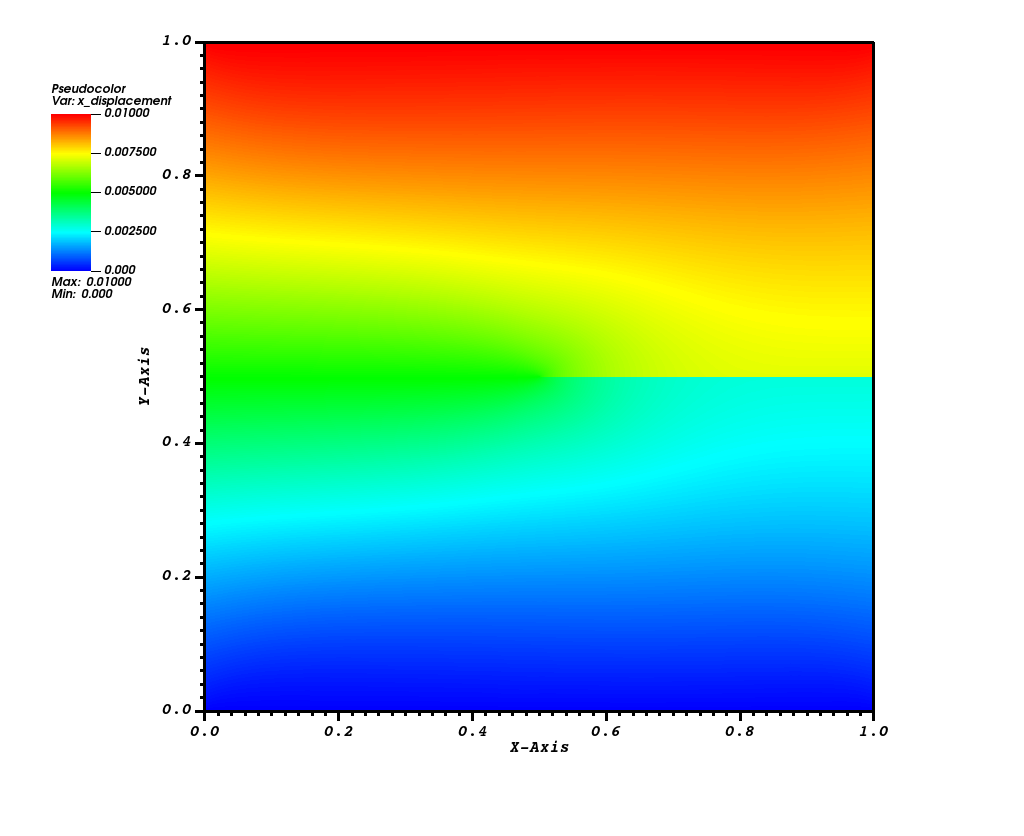}
        \caption{$\bfu_x$}
        \label{tensile_ux}
    \end{subfigure}
    \hfill 
    \begin{subfigure}[b]{0.48\textwidth}
        \centering
        \includegraphics[width=\textwidth]{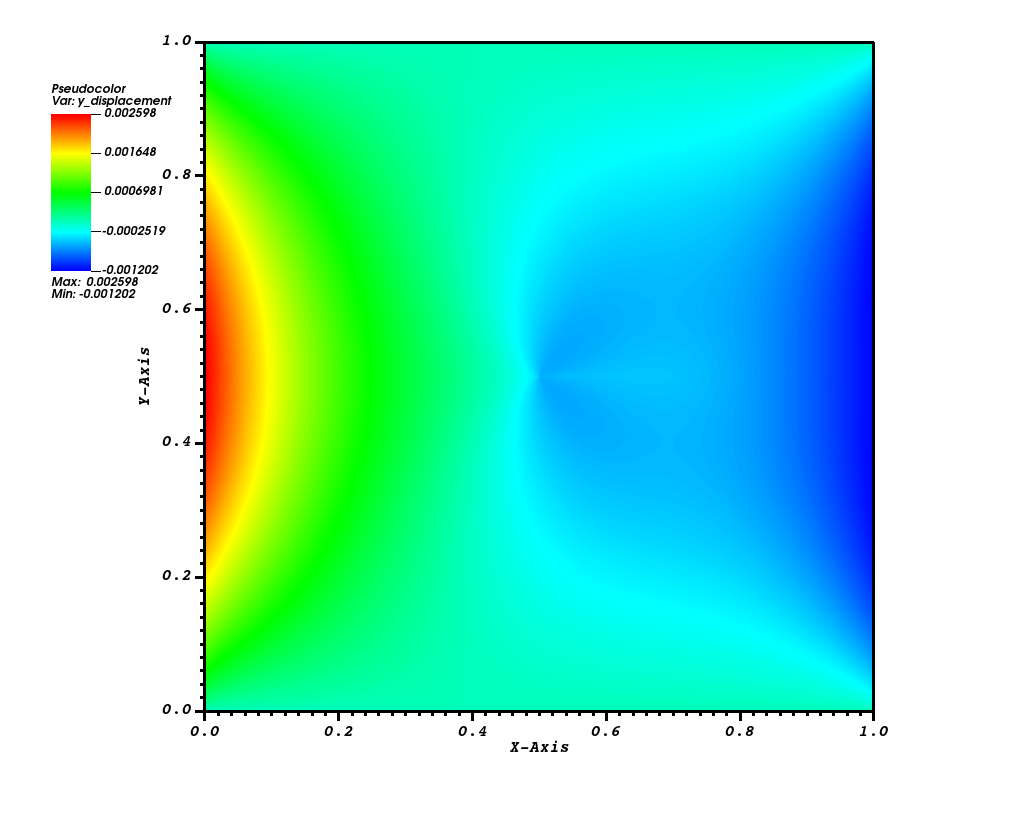}
        \caption{$\bfu_y$}
        \label{tensile_uy}
    \end{subfigure}  
    \caption{Displacement fields under pure shear (Mode II) loading. (a) illustrates the characteristic sliding discontinuity along the crack line in the horizontal direction, while (b) shows the continuous vertical deformation field, confirming the absence of crack opening. }
    \label{ux_uy_shear}
\end{figure}

The sensitivity of the shear-induced crack-tip fields to the nonlinearity parameter $\beta$ is examined in Figure \ref{beta_plot_shear}. The plots depict the variation of field variables along the crack ligament ($y=0.5$) as they approach the tip. A notable feature of this nonlinear model is the development of significant normal stress ($\bfT_{22}$) and vertical strain ($\bfeps_{22}$) components near the tip, even under pure shear boundary conditions. As shown in Figure \ref{beta_plot_shear}(b) and (c), these components exhibit a compressive (negative) singularity. Consistent with the tensile results, a negative $\beta$ (dashed lines) markedly amplifies this singularity, resulting in sharp compressive spikes and a corresponding surge in the Strain Energy Density (Figure \ref{beta_plot_shear}(d)). Conversely, positive values of $\beta$ (solid lines) act as a regularization mechanism, effectively suppressing the magnitude of the compressive stresses and strains, thereby reducing the localized energy concentration compared to the linear elastic baseline ($\beta=0$).

\begin{figure}[H]
\centering
\includegraphics[width=\textwidth]{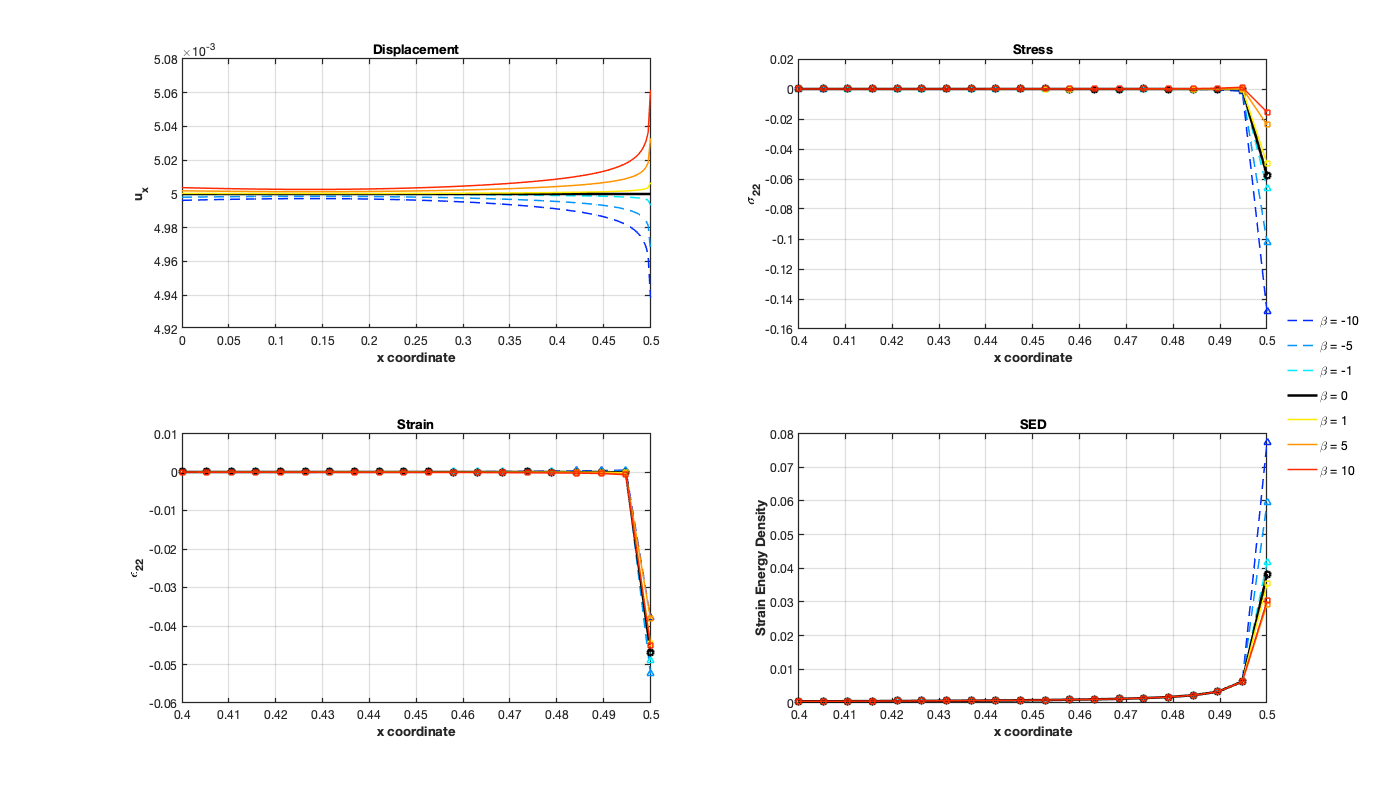} 
\caption{Effect of the nonlinearity parameter $\beta$ on the crack-tip fields under Mode II shear loading. The panels show (a) horizontal displacement $\bfu_x$, (b) vertical stress $\bfT_{22}$, (c) vertical strain $\bfeps_{22}$, and (d) SED along the ligament. Note the emergence of compressive (negative) normal stresses and strains near the tip. Negative $\beta$ values (dashed lines) intensify the singularity, leading to higher energy densities, while positive $\beta$ values (solid lines) mitigate the field concentrations.}
\label{beta_plot_shear}
\end{figure}

\subsection{Mixed-Mode Loading}
The discretization profile generated by the $hp$-adaptive algorithm under mixed-mode (combined tensile and shear) loading is shown in Figure \ref{mixed}. Figure \ref{mixed}(a) depicts the final $h$-refined mesh. The refinement pattern exhibits a distinct asymmetry compared to the pure mode cases; while the crack tip at $(0.5, 0.5)$ remains the primary focus of element subdivision to resolve the singularity, significant refinement also emerges at the top corners of the domain. This corner refinement is a direct response to the complex boundary conditions where both vertical and horizontal displacements are prescribed, creating localized stress concentrations that the error estimator seeks to resolve. Figure \ref{mixed}(b) illustrates the corresponding polynomial degree distribution. The algorithm assigns lower-order polynomials ($p=1$, blue/cyan) to the singular tip elements and the high-gradient corner regions to ensure stability, while effectively elevating the polynomial order ($p=4$ to $5$, yellow/orange) in the smoother bulk of the domain. This optimized distribution of $h$ and $p$ verifies the solver's capability to adaptively capture the interacting fields of tension and shear.

\begin{figure}[H]
    \centering
    \begin{subfigure}[b]{0.48\textwidth}
        \centering
        \includegraphics[width=\textwidth]{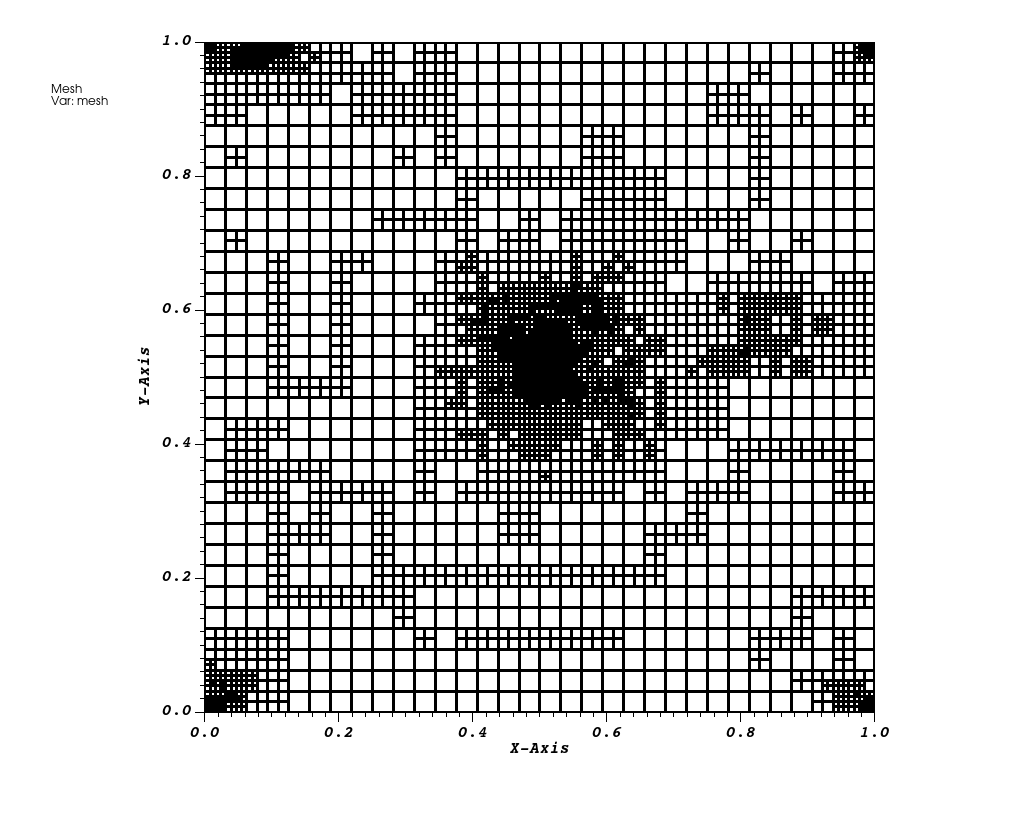}
        \caption{Mesh}
        \label{tensile_mesh}
    \end{subfigure}
    \hfill 
    \begin{subfigure}[b]{0.48\textwidth}
        \centering
        \includegraphics[width=\textwidth]{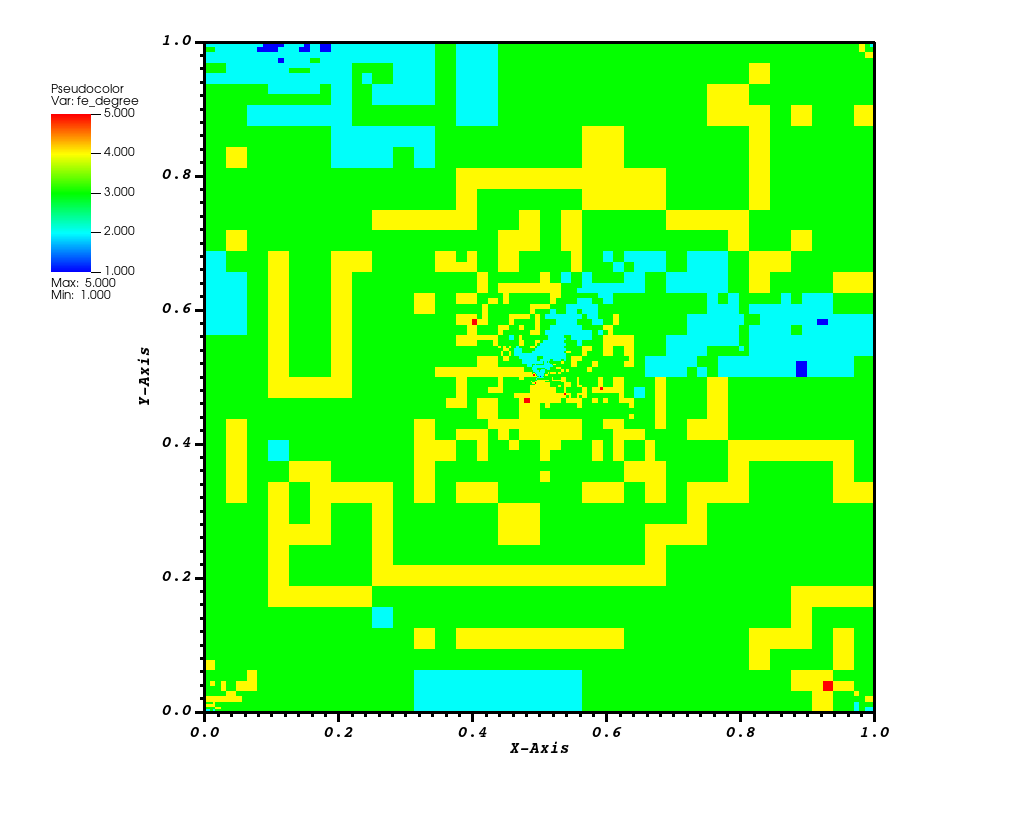}
        \caption{Polymomial degree}
        \label{fedeg_mixed}
    \end{subfigure}  
    \caption{Adaptive discretization results for the mixed-mode loading case. (a) shows the asymmetric mesh refinement driven by the combined kinematic constraints, while (b) shows the polynomial degree distribution, balancing low-order stability at the tip and corners with high-order accuracy in the bulk.}
    \label{mixed}
\end{figure}

The displacement fields resulting from the combined tensile and shear loading are presented in Figure \ref{fig:disp_mixed}. These results clearly manifest the superposition of the two fundamental kinematic modes associated with fracture. Figure \ref{fig:disp_mixed}(a) displays the horizontal displacement ($\bfu_x$), which exhibits a sharp discontinuity along the crack interface ($y=0.5$), indicative of the relative sliding (Mode II) driven by the horizontal shear component of the load. Simultaneously, Figure \ref{fig:disp_mixed}(b) reveals a distinct jump in the vertical displacement ($\bfu_y$) across the branch cut. Unlike the pure shear case, where the vertical field remained continuous, here the upper flank lifts significantly relative to the lower flank, representing the crack opening (Mode I) induced by the vertical tension. The ability of the numerical scheme to resolve sharp discontinuities in both coordinate directions simultaneously demonstrates its robustness in handling complex mixed-mode deformation patterns.

\begin{figure}[H]
    \centering
    \begin{subfigure}[b]{0.48\textwidth}
        \centering
        \includegraphics[width=\textwidth]{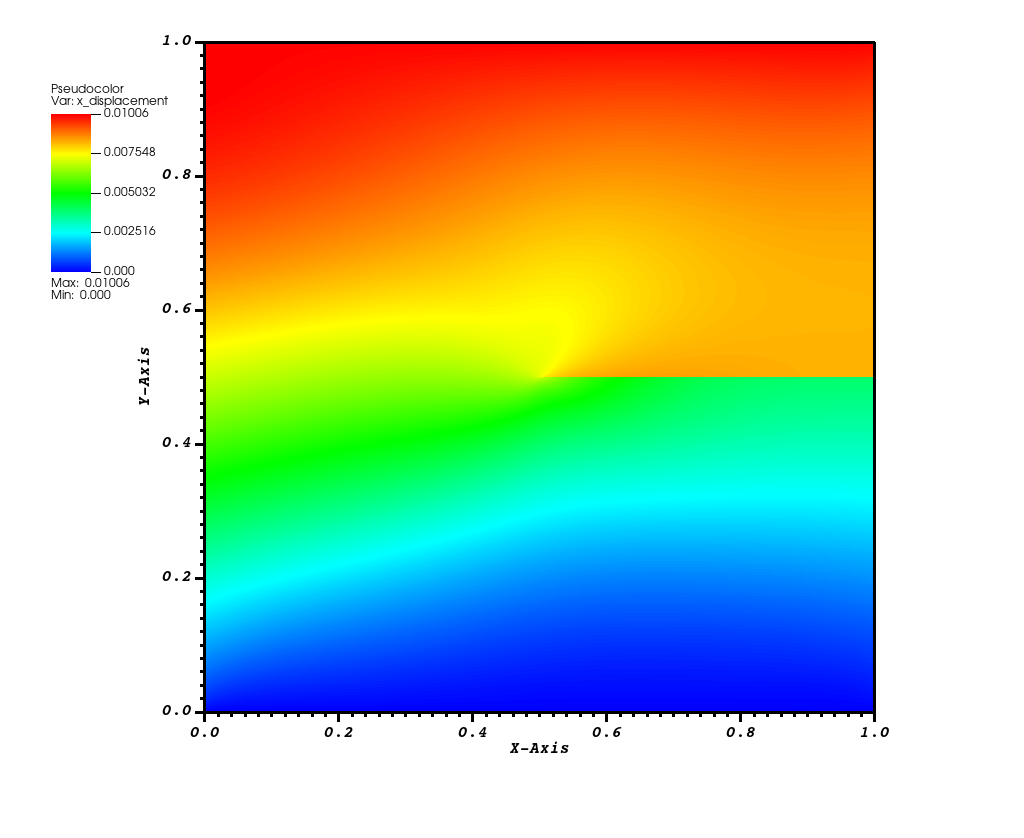}
        \caption{$\bfu_x$}
        \label{tensile_ux}
    \end{subfigure}
    \hfill 
    \begin{subfigure}[b]{0.48\textwidth}
        \centering
        \includegraphics[width=\textwidth]{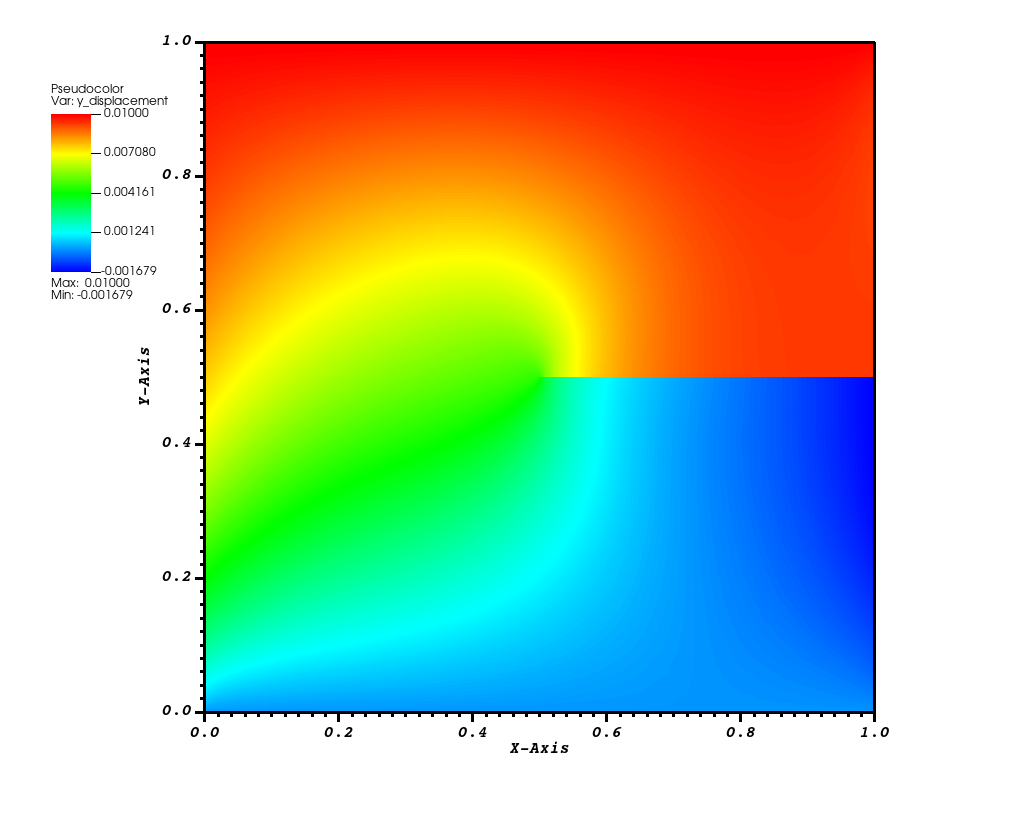}
        \caption{$\bfu_y$}
        \label{tensile_uy}
    \end{subfigure}  
   \caption{Computed displacement fields for the mixed-mode (combined tensile and shear) loading. (a) captures the sliding deformation along the crack characteristic of Mode II, while (b) captures the crack opening displacement characteristic of Mode I. The simultaneous presence of discontinuities in both fields confirms the mixed-mode nature of the response.}
    \label{fig:disp_mixed}
\end{figure}

The influence of the nonlinearity parameter $\beta$ on the crack-tip fields under mixed-mode (combined tensile and shear) loading is analyzed in Figure \ref{fig:field_variables_mixed}. The plots track the evolution of the horizontal displacement ($\bfu_x$), vertical stress ($\bfT_{22}$), vertical strain ($\bfeps_{22}$), and SED along the uncracked ligament ($y=0.5$) as they approach the crack tip at $x=0.5$. Similar to the pure tensile case, the results reveal a strong bifurcation in material response based on the sign of $\beta$. Negative values (dashed blue lines) significantly exacerbate the singularity, leading to peak stresses and strains that far exceed the linear elastic prediction ($\beta=0$, black line). This is particularly evident in the SED plot (Figure \ref{fig:field_variables_mixed}(d)), where the energy concentration for $\beta=-10$ is nearly double that of the pure tensile case, reflecting the compounded severity of the mixed-mode deformation. Conversely, positive values of $\beta$ (solid orange/red lines) continue to exhibit a regularizing behavior, effectively shielding the crack tip by lowering the peak intensities of both stress and strain.

\begin{figure}[H]
\centering
\includegraphics[width=\textwidth]{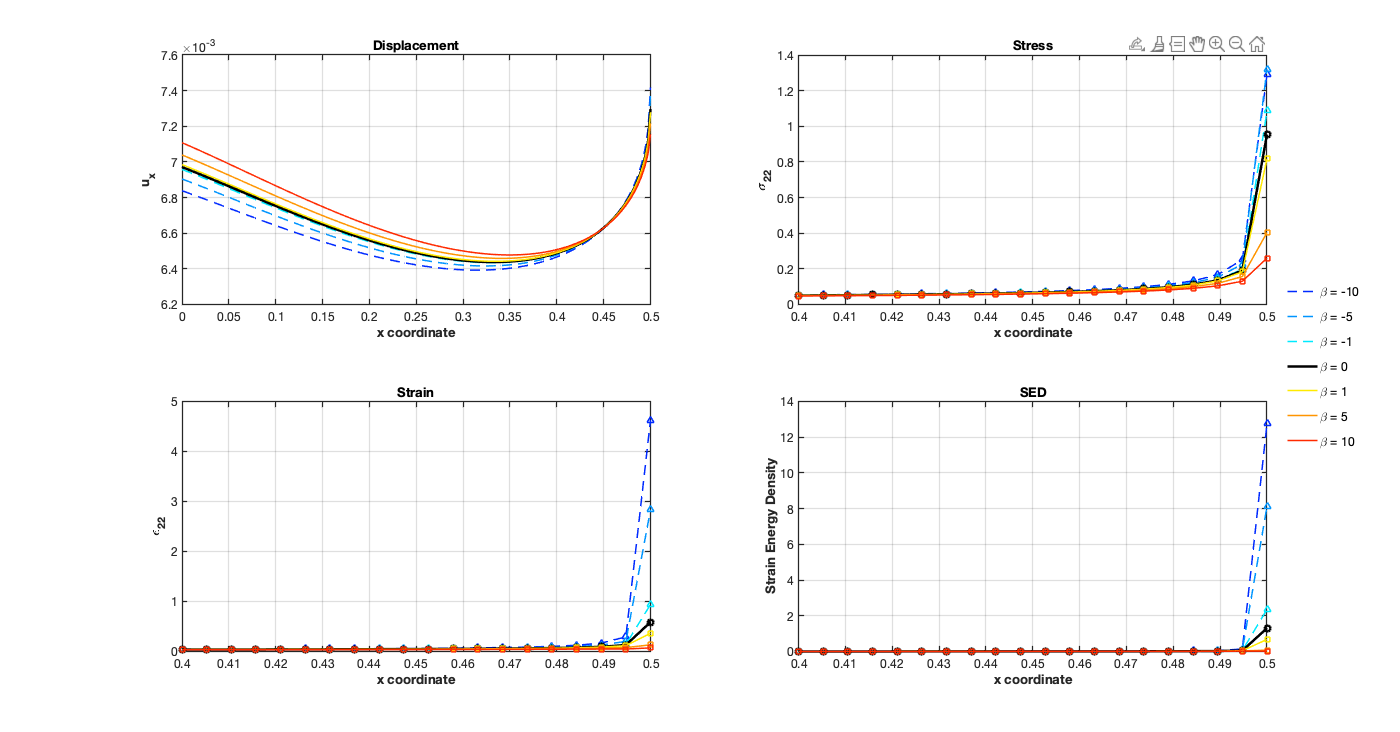} 
\caption{Distribution of field variables along the ligament ($y=0.5$) approaching the crack tip under mixed-mode loading. The panels display (a) horizontal displacement $\bfu_x$, (b) vertical stress $\bfT_{22}$, (c) vertical strain $\bfeps_{22}$, and (d) SED. The interaction of tension and shear results in higher energy concentrations compared to the pure modes. Consistent with previous cases, negative $\beta$ values (dashed lines) amplify the singularity, while positive $\beta$ values (solid lines) attenuate the near-tip fields relative to the linear elastic baseline.}
    \label{fig:field_variables_mixed}
\end{figure}

\section{Conclusion}
\label{sec:conclusion}

In this work, we have presented a rigorous mathematical and computational framework for analyzing static fracture in isotropic elastic solids governed by a class of nonlinear, density-dependent constitutive laws. The primary contribution of this study is the formulation of a well-posed boundary value problem that extends classical linear elasticity to account for material moduli that are intrinsically coupled to the volumetric strain. By establishing the ellipticity and monotonicity conditions of the stress tensor, we proved the existence and uniqueness of weak solutions within the appropriate Sobolev spaces, provided the material parameters and data satisfy specific boundedness constraints to avoid constitutive singularities.

To resolve the complex field gradients inherent to this nonlinear model, we developed and implemented a robust $hp$-adaptive finite element scheme. The solution strategy employed a damped Newton-Raphson algorithm augmented with a line search globalization method, ensuring quadratic convergence even in the presence of the localized stiffening or softening induced by the material nonlinearity.

Our numerical investigations into Mode I (tensile), Mode II (shear), and mixed-mode fracture scenarios yielded several key insights:
\begin{itemize}
    \item \textit{Regularization vs. singularity:} The nonlinearity parameter $\beta$ acts as a critical control variable. Positive values of $\beta$ were shown to regularize the crack-tip fields, effectively screening the singularity and reducing the strain energy density compared to the linear elastic baseline. Conversely, negative values of $\beta$ significantly amplified the singularity, leading to a profound concentration of stress and strain that necessitates the use of high-order polynomial approximations for accurate resolution.
    \item \textit{Algorithm robustness:} The $hp$-adaptive strategy proved essential for capturing the solution topology. The automated refinement successfully generated graded meshes with high-order polynomials ($p=5$ or $6$) in the smooth bulk regions and $h$-refinement with low-order stability ($p=1$) at the crack tip. This dual strategy allowed for the sharp resolution of displacement discontinuities across the crack faces while maintaining optimal convergence rates.
    \item \textit{Mixed-mode interaction:} The solver demonstrated the capability to capture complex, non-symmetric deformation patterns where crack opening and sliding modes coexist. The interaction of these modes under the nonlinear law resulted in energy densities that are not merely the superposition of the individual modes, highlighting the non-trivial coupling introduced by the volumetric nonlinearity.
\end{itemize}

The results presented herein lay the foundation for several promising avenues of future research, particularly regarding the temporal evolution of fracture in density-dependent media.
\begin{itemize}
\item \textit{Quasi-static evolution:} A natural extension is the study of quasi-static crack propagation, where inertial effects are negligible, but the crack length evolves as a function of a loading parameter. This introduces a free-boundary problem where the energy release rate must be re-evaluated in the context of the nonlinear constitutive law. The convexity properties of the strain energy density established in this work will be pivotal in determining the stability of crack paths and the energetic threshold for propagation.

\item \textit{Elastodynamics and hyperbolicity:} Extending this model to the dynamic regime poses significant mathematical challenges. The introduction of the inertia term transforms the governing equation into a nonlinear hyperbolic system. Investigating the formation and propagation of shock waves in such solids, and specifically how the density-dependent moduli influence wave speeds and dispersion near the crack tip, is of high interest. Furthermore, the well-posedness of the dynamic problem—specifically the conditions under which the system remains strictly hyperbolic—remains an open question that warrants detailed analysis \cite{larsen2010models}.

\item \textit{Phase-field regularization:} To obviate the need for explicit tracking of the crack geometry in complex evolution problems, coupling the current nonlinear elasticity model with a phase-field approach to fracture is a logical next step \cite{lee2022finite,yoon2021quasi,manohar2025adaptive}. This would require formulating a coupled energy functional that accounts for the competition between the bulk nonlinear elastic energy and the dissipative fracture energy, potentially leading to new insights into damage nucleation in heterogeneous materials.

\item \textit{Limitation of the current model:} While the proposed framework successfully elucidates the role of density-dependent nonlinearity in regularizing crack-tip singularities, it is important to acknowledge a specific kinematic limitation inherent to the current formulation. The variational problem defined in \eqref{weak_form_new} treats the crack faces as free boundaries subject to Neumann conditions (traction-free or prescribed load). This formulation implicitly assumes that the crack faces move apart (opening mode) or slide past one another without geometric interference. However, under compressive loading or specific mixed-mode scenarios where the crack closure effect is dominant, the current model does not mathematically preclude the interpenetration of the opposing crack faces. Physically, this represents a violation of the injectivity of the deformation map (i.e., matter overlapping). In the numerical results presented herein, this limitation was circumvented by restricting the investigation to loading regimes that induce strictly positive crack opening displacements. Nevertheless, for a fully generalizable fracture model capable of simulating contact, friction, and closure, this non-physical interpenetration must be rigorously addressed. To rectify this limitation in future developments, the boundary value problem must be reformulated from a variational equality to a \textit{variational inequality}. This involves imposing non-penetration constraints, often referred to as the \textit{Signorini contact conditions} (or unilateral contact conditions).

\end{itemize} 

\section*{Acknoledgement}
The work of SMM is supported by the National Science Foundation under Grant No.\ 2316905.

\bibliographystyle{plain}
\bibliography{references}

\end{document}